\documentclass[final,onefignum,onetabnum]{siamart190516}

\usepackage{lipsum}
\usepackage{amsfonts}
\usepackage{graphicx}
\usepackage{epstopdf}
\usepackage{algorithmic}
\ifpdf
  \DeclareGraphicsExtensions{.eps,.pdf,.png,.jpg}
\else
  \DeclareGraphicsExtensions{.eps}
\fi

\newsiamremark{remark}{Remark}
\newsiamremark{hypothesis}{Hypothesis}
\crefname{hypothesis}{Hypothesis}{Hypotheses}
\newsiamthm{claim}{Claim}

\newsiamthm{assumption}{Assumption}
\newsiamremark{example}{Example}

\title{Minimax Problems with Coupled Linear Constraints: \\
	Computational Complexity, Duality and Solution Methods}

\author{
Ioannis Tsaknakis\thanks{Department of Electrical and Computer Engineering, University of Minnesota, tsakn001@umn.edu}
\and
Mingyi Hong\thanks{Department of Electrical and Computer Engineering, University of Minnesota, mhong@umn.edu}
\and
Shuzhong Zhang\thanks{Department of Industrial and Systems Engineering, University of Minnesota, zhangs@umn.edu}
}

\usepackage{algorithmic}
\usepackage{algorithm}  
\usepackage{graphicx} 
\usepackage{subcaption}
\usepackage{caption}
\usepackage{enumitem}

\newcommand{\x}{\mathbf{x}}
\newcommand{\y}{\mathbf{y}}

\newcommand{\lb}{\boldsymbol{\lambda}}
\newcommand{\cb}{\mathbf{c}}

\newcommand{\X}{\mathcal{X}}
\newcommand{\Y}{\mathcal{Y}}
\newcommand{\R}{\mathbb{R}}
\newcommand{\xx}{\x \in \mathcal{X}}
\newcommand{\yy}{\y \in \mathcal{Y}}
\newcommand{\xo}{\overline{\x}}
\newcommand{\yo}{\overline{\y}}
\newcommand{\lo}{\overline{\lb}}
\newcommand{\xt}{\widetilde{\x}}
\newcommand{\yt}{\widetilde{\y}}
\newcommand{\lt}{\widetilde{\lb}}

\newcommand{\g}{\nabla}
\newcommand{\gx}{\nabla_{\x}}
\newcommand{\gy}{\nabla_{\y}}

\graphicspath{{figures/}}

\begin{document}

 \title{Minimax Problems with Coupled Linear Constraints: \\
 	Computational Complexity and Duality}

 \author{
 Ioannis Tsaknakis\thanks{Department of Electrical and Computer Engineering, University of Minnesota, tsakn001@umn.edu}
 \and
 Mingyi Hong\thanks{Department of Electrical and Computer Engineering, University of Minnesota, mhong@umn.edu}
 \and
 Shuzhong Zhang\thanks{Department of Industrial and Systems Engineering, University of Minnesota, zhangs@umn.edu}
 }

\maketitle


\begin{abstract}
In this work we study a special minimax problem where there are linear constraints that couple both the minimization and maximization decision variables. 
The problem is a generalization of the traditional saddle point problem (which does not have the coupling constraint), and it finds applications in wireless communication, game theory, transportation,  just to name a few. We show that the considered problem is challenging, in the sense that it violates the classical max-min inequality, 
and that it is NP-hard even under very strong assumptions (e.g., when the objective is strongly convex-strongly concave). We then develop a duality theory for it, and analyze conditions under which the duality gap becomes zero.
Finally, we study a class of stationary solutions defined based on the dual problem, and evaluate their  practical performance in an application on adversarial attacks on network flow problems.
\end{abstract}



\section{Introduction}

The \textit{minimax optimization problem}, given below: 
\begin{align}\label{eq:minmax}
\min_{\xx} \max_{\yy} f(\x,\y)
\end{align}
finds applications in areas such as machine learning,  game theory,   signal processing, and it has been extensively studied in  recent years \cite{nouiehed2019solving, lu2020hybrid,lin2020gradient, lin2020near, yang2020global}. 
Its specific applications include adversarial learning \cite{madry2017towards}, reinforcement learning \cite{qiu2020singletimescale}, 
resource allocation in wireless communication \cite{lu2020hybrid}, and Generative Adversarial Networks (GAN) \cite{goodfellow2014generative}, among others. 

In this work\footnote{ Due to space limitations we were unable to include all the results of our work in this manuscript. The interested readers can refer to an extended version of this paper which is available online \cite{tsaknakis2021minimax}.} we consider a class of more general minimax problems, where the constraint set of the inner problem linearly depends on {\it both} optimization variables:  
	\begin{align}\label{eq:minmax-coupled}
	\min\limits_{\x \in \X} 
	\left(\max\limits_{\y \in \Y, A\x+B\y\leq \bf{c}}\tag{mM-I} f(\x,\y)\right). 
	\end{align}
In the above expression, $f(\x,\y):\R^{n} \times \R^{m} \to \R$, $\X\subseteq \R^{n}, \Y \subseteq \R^{m}$, $A \in \R^{k \times n}, B \in \R^{k \times m}, c \in \R^{k}$.  Let us use $[A\x+B\y]_i\le \mathbf{c}_i$ to denote the $i$th constraint, and define $\mathcal{K}:=\{1,\cdots, k\}$ as the index set of the constraints. 
As a naming convention, in this work we will use `M' and `m' to denote the maximization and minimization problems, respectively. Therefore, we refer to the above problem as a {\it miniMax with Inner-level coupling (mM-I)} problem.  Further, the maximization variable is always $\y$, while the minimization variable is always $\x$.

 Problems \eqref{eq:minmax} and  \eqref{eq:minmax-coupled} appear to be closely related, so one might think that they have similar properties.
For example, the original problem \eqref{eq:minmax} is easily solvable when $f(\x,\y)$ is convex in $\x$ and concave in $\y$, so this may lead to the belief that the linearly constrained version is still relatively easy. 
However, we will show that problem \eqref{eq:minmax-coupled} is \textit{NP-hard} in general, even when $f(\x,\y)$ is strongly-convex  and strongly-concave. Further, the classical max-min inequality \cite[Theorem 1.3.1]{nesterov2018lectures} does not hold for such a class of problems. Finally, existing algorithms developed for minimax problems cannot be  applied directly because they can get stuck at some uninteresting solution points. 
On the other hand, compared with \eqref{eq:minmax}, these problems can be used to model a wider class of applications, a few of which are presented below.

\subsection{Representative Applications}\label{sec:apps}

\subsubsection{Adversarial attacks in resource allocation problems}\label{sec:res_alloc}

Consider a setting where a player (the `user') aims to optimally allocate a resource of fixed total amount $c_0 \in \R$, across $n$ different tasks. The goal is to maximize certain utility which is a function of the allocation. 
Also,  there is an adversary that tries to reduce the utility of the user, by designing 
an allocation strategy across the $n$ tasks that `forces' the user (which utilizes the remaining resource) to select an allocation that lowers its utility. As a concrete example, let us consider a transmission rate maximization problem under the presence of a jammer, which arises in wireless communications. 

Specifically, we consider a problem where a single user (i.e., a transmitter-receiver pair) transmits messages over $n$ channels. The goal of this user it to allocate power (the `resource') across the $n$ channels (the `tasks') such that its transmission rate (the `utility') is maximized. Suppose that there is a jammer (the `adversary') in the system, whose goal is to minimize the user's rate (e.g., \cite{gohary2009generalized}). 
In this problem we assume that the channel is Gaussian, and denote with $\x, \y$ the power allocations of the user and the jammer, respectively. Then, the user's system rate is given by \cite{gohary2009generalized,lu2020hybrid}
$$
R(\x,\y) = \sum\limits_{i=1}^{n} \log\left(1+ \frac{g^{u}_{i} x_{i}}{\sigma^{2} + g^{j}_{i} y_{i}} \right),
$$
where $g^{u}_{i}$, $g^{j}_{i}$ are the gains of the user and the jammer on channel $i$, respectively, and $\sigma^{2}$ is the noise power.
Note that both players' power allocations are subject to constraints $\X=\{\x \in \R^{n}\mid \x \ge {\bf{0}}, \sum_{i=1}^{n} x_i \leq \bar{x}\}$ and $\Y=\{\y \in \R^{n}\mid \y\ge{\bf{0}}, \sum_{i=1}^{n} y_i \leq\bar{y}\}$, where $\bar{x},\bar{y}$ are the total power budgets for the user and the adversary, respectively. 
Moreover, in certain communication systems such as cognitive radio network \cite{yang2013robust}, it is required that the total transmission power on each channel is upper bounded by the so-called {\it interference temperature} (denoted by a constant $c>0$), in order to limit the {\it total} interference caused to other (perhaps more important) co-channel users in the system.  
Consequently, the problem of how to {\it optimally}  attack the user in such a multi-channel wireless system can be formulated as the following linearly constrained minimax problem, which is a special case of \eqref{eq:minmax-coupled}:
\begin{align*}
\min\limits_{\y \in \Y} 
\left(\max\limits_{\x \in \X, \x+\y \leq \bf{c}} R(\x,\y)\right). 
\end{align*}

\subsubsection{Adversarial attacks in network flow problems}\label{sec:net_flow}

Consider a \textit{flow network} represented by a directed graph $G=(V,E)$, where $V$ is the set of vertices and $E$ is the set of edges. Let  $x_e$  and $p_e$ denote the flow and capacity on each edge $e\in E$, respectively; let $\x:=\{x_e\}_{e\in E} \in \R^{|E|}$ and $\mathbf{p}:=\{p_e\}_{e\in E} \in \R^{|E|}$ denote the vectors of edge flows and the edge capacities, respectively. 
Let  $q_{e}(x_{e})$ denote the cost of moving one unit of flow across edge $e$,  
and  $s$ and  $t$  the source and sink node, respectively.  
Suppose $F$ is a set that collects all the edges used by $\x$ to deliver a total of $r_t$ units of flow from $s$ to $t$. Then  the total transportation cost is defined as $q_{tot}(\x) = \sum_{e \in F} q_{e}(x_{e}) x_{e}$. 

A minimum cost network flow problem can be defined as finding the paths with the minimum cost from source to sink, so that we can successfully transport certain amount of flow \cite{cormen2009introduction}. Let us consider an extension to such a problem, where in addition to the regular network user, there is an adversary who will inject flows to the network to force the regular user to use more expensive paths. Let $\y:=\{y_e\}_{e\in E}$ denote the set of flows controlled by the adversary, and $b>0$ be its total budget. Then, the problem of adversary can be formulated as follows:
\begin{align}\label{eq:net_atttack}
	\max\limits_{\stackrel{\bf{0} \leq \bf{y} \leq \bf{p}}{\sum_{(i,j) \in E} y_{ij} = b}} 
	 & \min \limits_{\stackrel{\bf{0} \leq \bf{x} \leq \bf{p}}{\sum\limits_{(i,t) \in E} x_{it} = r_t} } \sum_{(i,j) \in E} \overline{q}_{ij}(x_{ij}, y_{ij})x_{ij}  \\
	 \text{s.t. } & \bf{x} + \bf{y} \leq \bf{p} \nonumber \\
	 &\hspace{-3mm} \sum_{(i,j) \in E} x_{ij} - \sum\limits_{(j,k) \in E} x_{jk}= 0, \; \forall j \in V \setminus \{s,t\}, \nonumber
\end{align}
where $\overline{q}_{ij}(x_{ij}, y_{ij})$ is related to the cost of routing both $x_{ij}$ and $y_{ij}$ onto the network; a simple choice is $\overline{q}_{ij}(x_{ij}, y_{ij}) = q_{ij}(x_{ij}+y_{ij})$. 
We note that such a kind of  attack falls under the scope of network interdiction problems \cite{smith2013modern} and there are several works in literature that study variants of this application \cite{fu2019network, salmeron2004analysis, salmeron2009worst}. We note that problem \eqref{eq:net_atttack} can be used to model attacks for real systems such as  
 communication networks \cite{fu2019network} or power networks \cite{salmeron2004analysis}, in which the adversary aims to maximize the cost of the network owner by injecting spurious traffic or disabling network components.

\subsection{Contributions}

In this work, we study the minimax problem with coupled linear constraints \eqref{eq:minmax-coupled}. Our main contributions are listed below:

\noindent $\bullet$ We 
 first identify relationships between \eqref{eq:minmax-coupled} and a number of its variants, and  show that even when $f(\x,\y)$ is strongly-convex in $\x$ and strongly-concave in $\y$, the well-known max-min inequality does not hold true. Additionally, we show that  \eqref{eq:minmax-coupled} and its variants are NP-hard in general.
 
\noindent $\bullet$ 
We  develop a duality theory under the assumption that the objective $f$ is strongly concave w.r.t $\y$. Specifically, we define three different dual problems, identify the conditions under which strong duality holds, and establish the equivalence of their solutions to that of the original primal problem. 

 \noindent $\bullet$ Based on our developed dual problems, we identify a new stationary solution concept, and develop a first-order algorithm to evaluate the practical performance of such a solution concept on an adversarial attack problem for network flows.

 We emphasize that in order to evaluate the quality of the proposed stationary solution concept, we need to develop an efficient numerical algorithm. The algorithm itself is new (to our knowledge),  but its analysis is relatively standard, and we do not have the space to include  the details here; the algorithmic development is indeed not the focus of this paper. 

\subsection{Related Works}\label{sec:related}

The class of problems \eqref{eq:minmax-coupled} is closely related to problems in optimization and game theory. Below, we review the related literature.  

\smallskip
\noindent{\bf Minimax problems.} These problems have been extensively studied  
under different assumptions for the objective $f$. For instance, there are works for (strongly)-convex (strongly)-concave \cite{mokhtari2019convergence, mokhtari2019unified, yan2020optimal}, non-convex (strongly)-concave \cite{lin2020gradient, lin2020near, lu2020hybrid, nouiehed2019solving}, and non-convex non-concave \cite{liu2019better, nouiehed2019solving, yang2020global} problems. We refer the interested readers to a recent survey about detailed developments \cite{razaviyayn2020nonconvex}.  Specifically, in \cite{lin2020gradient} the `classical' Gradient Descent-Ascent (GDA) algorithm is analyzed, and convergence is established to stationary solutions, under the assumption that the objective is non-convex (strongly)-concave. Moreover, in \cite{mokhtari2019unified} two extensions of GDA are studied, namely the Optimistic Gradient Descent-Ascent (OGDA), and the Extragradient (EG) algorithm, albeit in a convex-concave setting. In another extension, the inner-level problem is solved using multiple steps of a first-order algorithm; the convergence of this scheme is analyzed in \cite{nouiehed2019solving} assuming that the objective is non-convex in $\x$, and satisfies either the PL condition or is concave in $\y$.

The above works focus on problems {\it without} any constraints that couple the inner and outer variables. In these problems, the optimality conditions that are typically used are: a) the stationarity conditions of objective $f$ w.r.t $\x$ and $\y$, and b) the stationarity condition of the upper-level objective $\rho(\x) = \max_{\yy} f(\x,\y)$, i.e., $0 \in \partial \rho(\x)$. However, these conditions were defined for problems with no coupled constraints, and thus they cannot be applied in our case. Therefore, it is no longer clear how algorithms such as OGDA can be used. 
Consequently, the analysis of \eqref{eq:minmax-coupled} requires different approaches. 

There are a couple of very recent works that study min-max problems with coupled constraints similar to \eqref{eq:minmax-coupled}. In \cite{dai2020optimality} the authors extend the notion of local minimax points (e.g. \cite{jin2019local}) for min-max problems with coupled constraints, and  derive the respective necessary and sufficient optimality conditions. However, no duality theory or algorithms are developed.
Moreover, in \cite{goktas2021convexconcave} the authors study a min-max problem where the objective is convex-concave, the coupled constraints are of the form $g(\x, \y) \geq 0$, for some function $g(\x, \y)$ that is concave in $\y$, and the solution concept they consider is the same as in this work. Most importantly they also assume that the inner function $\phi(\x) = \max_{\y \in \Y, g(\x,\y)\geq\bf{0}} f(\x,\y)$ is convex, which makes the problem tractable. On the contrary, we make no such assumption, and develop our theory and algorithm for the case where the problem is still difficult (i.e., NP-hard).

\smallskip
\noindent{\bf Generalized Nash equilibrium games.} Our problem formulation is also related to the \textit{generalized Nash equilibrium (GNE)} game \cite{facchinei2010generalized, dreves2011solution, fischer2014generalized}, which is an extension of Nash games where the action space of each player depends on the actions of the other players. Specifically, a two player  GNE game can be formulated as follows \cite{facchinei2010generalized}:
\begin{align}\label{eq:gne}
\min\limits_{\x \in \X (\y)} f_1(\x,\y), \;\; \min\limits_{\y \in \Y (\x)} f_2(\x,\y),
\end{align}
where $f_1$, $f_2$ are the utilities of two players.
Problem \eqref{eq:minmax-coupled} is related to GNE games in the zero-sum case where $f_1 = -f_2$. However, in GNE games  the goal is to attain a Nash equilibrium, and not a minimax/maximin point. In other words, unlike problems \eqref{eq:minmax-coupled} the two problems do not have an order. In addition, note that for GNE games, the equilibirum solutions do not always exist, and certain (strong) assumptions are required; for instance the utilities $f_1, f_2$ need to be quasi-convex \cite{ichiishi2014game}.

One technique that was developed for GNEs \cite{von2009optimization} utilizes the Nikaido-Isoda function \cite{nikaido1955note} to reformulate the GNE-finding problem to a constraint optimization task. However, the main results in this case have been obtained under the assumption that the utilities are convex.
Also, a number of penalty approaches \cite{pang2005quasi, facchinei2010penalty} have been proposed, nonetheless these are considered difficult to solve in practice. Another approach reduces the problem of finding the KKT points of a GNE game to a variational inequality problem \cite{aussel2008generalized}.
Overall, the above approaches cannot be applied in problem  \eqref{eq:minmax-coupled} because they were developed for problems of the form \eqref{eq:gne}, where each player has its own utility function, and the order of play does not matter.

\smallskip
\noindent{\bf Stackelberg games.} Another class of game-theoretic problems that are closely related to problem \eqref{eq:minmax-coupled} are \textit{Stackelberg games} \cite{von1952theory, osborne1994course, wang2019solving}, in which there are two players, the leader and the follower. The leader  can anticipate the follower's actions. In this case, the order in which the two players act matters.
Then a (general-sum) Stackelberg game can be formulated as the following bi-level program:
\begin{align*}
\min\limits_{\x \in \X} F(\x,\y) 
\text{ s.t. } \y \in \arg\min\limits_{\y \in \Y} f(\x,\y).   
\end{align*}
Therefore, these games are in some sense more general than problem \eqref{eq:minmax-coupled}, which can only model zero-sum Stackelberg games.
On the other hand, typically in Stackelberg games the action space of the follower (resp. the  leader) is independent of the action of the leader (resp.\ the follower); see, e.g.,  \cite[sec.\ 2.6.1]{zhang2015multi}. As a result,  problem \eqref{eq:minmax-coupled} extends the zero-sum Stackelberg games by taking into consideration the interactions between the players' actions. 

\smallskip
\noindent{\bf Bi-level optimization.} Finally, we would like to discuss the relationship between problem \eqref{eq:minmax-coupled} and bi-level optimization problems of the following form:
\begin{align*}
\min\limits_{\x \in \X} F(\x,\y) \quad 
\text{ s.t. } \y \in \arg\min\limits_{\y \in \Y(\x)} f(\x,\y).   
\end{align*}
Bi-level optimization  was formally introduced in \cite{bracken1973mathematical}, and it is also related to the broader class of problems of Mathematical Programming with Equilibrium Constraints \cite{luo_pang_ralph_1996}.  A generic instance of bi-level problems includes all minimax problems as special cases, with or without coupled constraint; see \cite{colson2005bilevel, dempe2020bilevel, liu2021investigating} for a number of survey papers. However, bi-level optimization problems are in general very challenging to solve. More precisely \cite{liu2021investigating}: 1) even linear bi-level problems are NP-hard, 2) the lower-level problem might have multiple solutions, 3) the feasible region defined by the lower-level problem can be a non-convex set, 4) they are non-smooth in general.

Several techniques have been developed in the literature for solving these problems. One such technique uses the value function $V(\x) = \max_{\y \in \Y(\x)} f(\x,\y)$ to express the lower-level problem as an inequality constraint, and transforms the bi-level problem into a single-level optimization task \cite{ye1995optimality, lin2014solving}. However, not only this reformulated problem does not satisfy any of the known constraint qualifications (CQ), but also the non-smoothness of $V(\x)$ makes the problem difficult. 
Alternative CQs such as calmness conditions have been proposed, 
which are trivially satisfied in minimax problems \cite{ye1995optimality}. Based on the calmness condition, a penalty method has been developed in \cite{lin2014solving}, but it only applies to bi-level problems in which the  constraints set of the lower-level problem is independent of the variable of the upper one.
Overall, this generic theory does not provide much understanding about the structure of our problem.

There is a recent line of work in (stochastic) single and two-time scale gradient descent methods \cite{ghadimi2018approximation, hong2020two, khanduri2021momentumassisted, chen2021singletimescale}. Under the assumption that the lower-level problem is strongly-convex (and thus admits a unique solution) and unconstrained, the authors attempt to minimize the objective $\phi(\x) = F(\x,\y^{\ast}(\x))$. In this case the implicit function theorem provides access to the gradient $\g \phi(\x)$. 
Unfortunately, in our case the constraints on the lower-level problem prevent us from using these type of methods.

All the above approaches were developed for general bi-level problems. As a result, they were designed to be suitable for a wide range of problems. Moreover, these approaches are either not applicable to our problem, or even if they are, they are not very attractive (e.g., not very efficient, do not have iteration complexity guarantees, difficult to analyze due to non-smoothness). On the other hand, problem \eqref{eq:minmax-coupled} has a special minimax structure, that is the lower-level objective $f$ is the negative of the upper level one $F$, and the coupling constraints are linear in the decision variables.  Therefore, in contrast to the general methods presented here, our goal is to leverage the special structure of problem \eqref{eq:minmax-coupled} to design methods tailored for it.

\section{Minimax problems with coupled linear constraints}\label{sec:coupled}

In this section, we derive a few unique properties about problem \eqref{eq:minmax-coupled}, and outline the challenges in designing efficient algorithms for it.
To begin with our analysis, let us first introduce a number of variants of problem of \eqref{eq:minmax-coupled}, and understand their relations. 

\begin{definition}%
\label{def:prob_coup_lin}
We define the following problems:
\begin{enumerate}[label=\Alph*.]
\item\label{prob_inn} Inner-Level Coupling (I)
\begin{enumerate}[label=\arabic*.]
\item miniMax with Inner-Level Coupling (mM-I)
\begin{equation}\tag{mM-I}\label{eq:inn_minimax}
\min\limits_{\x \in \X} 
\left( \max\limits_{\y \in \Y, A\x+B\y\leq\bf{c}} f(\x,\y) \right).
\end{equation}

\item Maximin with Inner-Level Coupling (Mm-I)
\begin{equation}\tag{Mm-I}\label{eq:inn_maximin}
\max\limits_{\y \in \Y} 
\left( \min\limits_{\x \in \X, A\x+B\y\leq\bf{c}} f(\x,\y) \right).
\end{equation}
\end{enumerate} 

\item\label{prob_out} Outer-Level Coupling (O)
\begin{enumerate}[label=\arabic*.]
\item miniMax with Outer-Level coupling (mM-O)
\begin{align}\tag{mM-O}\label{eq:out_minimax}
&\min\limits_{\x \in \X, A\x+B\y^*(\x) \leq \bf{c}} 
\left(\max\limits_{\y \in \Y} f(\x,\y) \right), \;\; \y^*(\x) \in \arg \max_{\y \in \Y} f(\x,\y) \nonumber.
\end{align}

\item Maximin with Outer-Level coupling (Mm-O)
\begin{align}\tag{Mm-O}\label{eq:out_maximin}
&\max\limits_{\y \in \Y, A\x^*(\y)+B\y \leq \bf{c}} 
\left(\min\limits_{\x \in \X} f(\x,\y) \right), \;\; \x^*(\y) \in \arg \min_{\x \in \X} f(\x,\y) \nonumber.
\end{align}
\end{enumerate} 

\end{enumerate}
\end{definition}


 We note that for the two outer-level coupled problems, the constraint set of the outer problem depends on the solutions $\y^*(\x)$ of the inner one. Since  $\y^*(\x)$ is typically a non-linear function, technically such a constraint is not a linear one any more.  Further, at this point these problems are under-specified, as it is not clear which of the solutions of $\y^*(\x)$ and $\x^*(\y)$ (e.g. one, some, or all) are expected to satisfy the constraints $A\x+B\y^*(\x) \leq \bf{c}$ and $A\x^*(\y)+B\y \leq \bf{c}$, respectively.
 Note that these two outer-level coupling problems are mainly introduced to better understand the two inner-level coupling problems. They have their own special structures, therefore require a separate treatment (which is beyond the scope of the current paper). 
 
To proceed, let us define the global minimax (or maxmin) solutions for these problems.   
For problem \eqref{eq:inn_minimax}, a point $(\x^*,\y^*)$ is the solution, if the following holds:
\begin{subequations}\label{eq:opt:inner}
\begin{align}
  &\x^* \in \arg\min_{\xx} \phi(\x), \quad \mbox{\rm where}\quad  \phi(\x) : = \max_{\y \in \Y, A\x+B\y\leq\bf{c}} f(\x,\y) \\
  &\y^* \in \arg\max_{\y \in \Y, A\x^*+B\y\leq\bf{c}} f(\x^*,\y).
  \label{eq:opt:inner2}
\end{align}
\end{subequations}

 For problem \eqref{eq:inn_maximin}, the global maxmin solution can be defined similarly.
Moreover, a solution $(\x^*,\y^*)$ of \eqref{eq:out_minimax} is defined as:
\begin{subequations}\label{eq:opt:outer}
\begin{align}
   &\y^* \in \arg\max_{\y \in \Y} f(\x^*,\y)\\
   &\x^* \in \arg\min_{\xx, A\x + B\y^* \leq \cb} \tilde{\phi}(\x), \quad \mbox{\rm where}\quad  \tilde{\phi}(\x) : = \max_{\y \in \Y} f(\x,\y).\label{eq:opt:outer2} 
\end{align}
   \end{subequations}
The solution of problem \eqref{eq:out_maximin} can be defined similarly.
Next, we provide necessary conditions for the existence of those solutions.

\begin{assumption}\label{as:feas}
Suppose the following conditions hold true. 
\begin{enumerate}[label=\alph*.]

\item \label{ass:cont} The function $f(\x,\y):\R^{n} \times \R^{m} \to \R$ is continuous w.r.t. both $\x$ and $\y$.

\item \label{ass:compact} The sets $\X$, $\Y$ are convex and compact, with $\|\x\| \leq D, \; \forall \x \in \X, \|\y\| \leq D, \;\forall \y \in \Y$, for some constant $D>0$.

\item \label{ass:feasibility} (Feasibility) The following feasibility conditions hold: 
\begin{itemize}
\item In problem \eqref{eq:inn_minimax}, 
 $\forall \xx, \exists \yy$, such that $A\x + B\y - \cb \leq {\bf{0}}$.
\item In problem \eqref{eq:inn_maximin}, $ \forall \yy, \exists \xx$, such that $A\x + B\y - \cb \leq {\bf{0}}$.
\item In problem \eqref{eq:out_minimax},
$\exists \xx$  such that $A\x + B\y^*(\x) - \cb \leq {\bf{0}}$ holds, for at least one solution $\y^*(\x)$.
\item In problem \eqref{eq:out_maximin}, $\exists \yy$  such that $A\x^*(\y) + B\y - \cb \leq {\bf{0}}$ holds, for at least one solution $\x^*(\y)$.
\end{itemize} 

\end{enumerate}
\end{assumption}

\begin{remark}
Under Assumption \ref{as:feas}, the inner-level tasks introduced in Definition \ref{def:prob_coup_lin} are well-defined.  
Specifically,  $\forall \xx$ in problems \eqref{eq:inn_minimax}, \eqref{eq:out_minimax}, the optimal solution of the inner problem exists, (or $\forall \yy$ in problems \eqref{eq:inn_maximin}, \eqref{eq:out_maximin}, the same holds). This is a result of Assumptions in \ref{as:feas}.\ref{ass:feasibility}, which ensure the feasibility of the inner problem, and Assumptions \ref{as:feas}.\ref{ass:cont}, \ref{as:feas}.\ref{ass:compact}, which guarantee the existence of a global solution, through Weierstrass' theorem. 
In addition, Berge's theorem of the maximum \cite[Theorem 17.31]{aliprantis1999infinite} implies that $\phi(\x)$ and $\tilde{\phi}(\x)$ are continuous functions.  This continuity property combined the compactness of the constraint sets, ensure the existence of global solutions in problems \eqref{eq:opt:inner2} and \eqref{eq:opt:outer2}, through  Weierstrass' theorem.
\end{remark}

Under Assumption \ref{as:feas}, 
it is easy to argue that
$f$ is lower/upper-bounded over $\X \times \Y$. That is, there exist finite constants $\underline{f}, \overline{f}$ such that:
\begin{align}\label{eq:bound}
 \underline{f} \le  f(\x,\y) \le \overline{f}, \; \forall~\x \in X, \; \y\in Y.
\end{align}%
Next, we 
characterize the relationships between the four problems in Definition \ref{def:prob_coup_lin}. As a convention, we will denote the {\it optimal} objective value of a problem $P$ as $v \left(P \right)$.

\begin{proposition}[Relations between problems]\label{prop:prob_relat}
For the values of the problems \eqref{eq:inn_minimax}, \eqref{eq:inn_maximin}, \eqref{eq:out_minimax} and  \eqref{eq:out_maximin}, the following relationships hold:

\begin{itemize}
\item $v (\mbox{\rm \ref{eq:inn_minimax}}) \leq v (\mbox{\rm \ref{eq:out_minimax}})$;

\item $v (\mbox{\rm \ref{eq:inn_maximin}}) \geq v (\mbox{\rm \ref{eq:out_maximin}})$;

\item $v (\mbox{\rm \ref{eq:out_minimax}}) \geq 
v (\mbox{\rm\ref{eq:out_maximin}})$.
\end{itemize}
Moreover, there is no definitive relation between the rest of the  pairs, namely the pairs 
$\left\{v (\mbox{\rm \ref{eq:inn_minimax}}), 
v (\mbox{\rm \ref{eq:inn_maximin})} \right\}$,
$\left\{v (\mbox{\rm \ref{eq:inn_minimax}}), 
v (\mbox{\rm \ref{eq:out_maximin})} \right\}$,
$\left\{v \mbox{\rm (\ref{eq:out_minimax}}), 
v (\mbox{\rm \ref{eq:inn_maximin}})\right\}$. That is, 
any relation of $<$, $=$, or $>$ can hold between them.
\end{proposition}
{\it Proof.} See Appendix \ref{app:minimax}.  

The most remarkable fact concerns the relation between  \eqref{eq:inn_minimax} and  \eqref{eq:out_maximin} [or between \eqref{eq:out_minimax} and \eqref{eq:inn_maximin}]. 
Note that after exchanging the min and max operators, \eqref{eq:out_maximin} becomes \eqref{eq:inn_minimax}.
Proposition \ref{prop:prob_relat} shows that that there is no definite relation  between \eqref{eq:inn_minimax} and  \eqref{eq:out_maximin}. 
On the contrary,  
in standard  minimax problems \eqref{eq:minmax}, the max-min inequality \cite[Theorem 1.3.1]{nesterov2018lectures} holds, that is
\begin{align}\label{eq:max-min-inequality}
    \max\limits_{\yy} \left( \min\limits_{\xx} f(\x,\y) \right)
\leq  \min\limits_{\xx} \left(\max\limits_{\yy} f(\x,\y) \right),
\end{align}
for any function $f$ and non-empty, closed sets $\X, \Y$. In addition, if $f$ is convex in $\x$ and concave in $\y$, and $\mathcal{X}, \mathcal{Y}$ are compact, then the values of these problems are equal.
 
The results shown in proposition \eqref{prop:prob_relat} indicate that the four problems we studied above should be harder than the standard minimax problems. Indeed, as we show below,  problems \eqref{eq:inn_minimax}-\eqref{eq:out_minimax} are all NP-hard in general, even when the objective function $f$ is strongly-convex strongly-concave, and the sets $\X$ and $\Y$ are compact.

\begin{proposition}\label{pro:np_hard}
Consider the \eqref{eq:inn_minimax} problem with a strongly-convex strongly-concave objective $f(\x,\y)$. This problem is NP-hard.
\end{proposition}
\begin{proof}
Consider the problem
\begin{align}\label{eq:prob_hard}
\min\limits_{\mathbf{0} \leq \x \leq \mathbf{1}} 
\left(\max\limits_{\stackrel{ -\mathbf{1} \leq \y \leq \mathbf{1}}{\x-2\y = \bf{0}}} \|\x\|^{2} + \frac{1}{2}\x^{T} Q \y - 4\|\y\|^{2} + \mathbf{d}^{T}\x \right).
\end{align}
where $Q \in \R^{n \times n}$ is a symmetric matrix with $Q \preceq 0$, $\mathbf{d} \in \R^{n}$, and $\mathbf{1}$ is a vector of ones. We can easily verify that problem \eqref{eq:prob_hard} is of the form \eqref{eq:inn_minimax}, and satisfies Assumption \ref{as:feas}. 
Then, using the coupled equality constraint  $\x-2\y = {\bf{0}}$ (which can be also expressed with two inequality constraints, i.e., $\x-2\y \leq {\bf{0}}$ and $-\x+2\y \leq {\bf{0}}$), problem  \eqref{eq:prob_hard} can be written as 
\begin{align*}
\min\limits_{\mathbf{0} \leq \x \leq \mathbf{1}}   \frac{1}{4}\x^{T} Q \x + \mathbf{d}^{T}\x.
\end{align*}
Since $Q \preceq 0$ the above is a difficult constrained concave minimization problem. To be more precise, it is established in \cite{pardalos1991quadratic} that this box-constrained non-convex quadratic problem is NP-hard. The proof is now completed.
\end{proof}

\begin{proposition}\label{pro:np_hard2}
Consider the \eqref{eq:out_minimax} problem with a strongly-convex concave objective $f(\x,\y)$. This problem is NP-hard.
\end{proposition}
\begin{proof}
Consider the problem
\begin{align}\label{eq:prob_hard2}
\min\limits_{\|\x \| \leq 1, \x-\y^{*}(\x)=0} 
\left(\max\limits_{ \|\y\| \leq 1} 
\left\{M \x^{T}\x + \sum\limits_{i,j}^{n} c_{ij}\x_{i}^{2}\x_{j}^{2} +\x^{T}\y  \right\} \right),
\end{align}
where $c_{i,j}$ are constants, and $M>0$ is selected such that the function $g(\x)= M \x^{T}\x + \sum\limits_{i,j}^{n} c_{ij}\x_{i}^{2}\x_{j}^{2}$ is strongly convex in a compact domain (which is a superset of the set $\|\x\| \leq 1$). Therefore, the  objective of problem \eqref{eq:prob_hard2} is strongly convex in $\x$, and concave (linear) in $\y$.
Then, \eqref{eq:prob_hard2} can be equivalently written as
\begin{align}\label{eq:np_hard_prf1}
\min\limits_{\|\x \| \leq 1, \x-\frac{\x}{\|\x\|}=0} 
\left\{ M \x^{T}\x + \sum\limits_{i,j}^{n} c_{ij}\x_{i}^{2}\x_{j}^{2} + \|\x\| \right\} 
\Leftrightarrow &\min\limits_{\|\x \| = 1} 
\left\{\sum\limits_{i,j}^{n} c_{ij}\x_{i}^{2}\x_{j}^{2} \right\} +M+1.
\end{align}
It is a known fact that the problem in the RHS of equation \eqref{eq:np_hard_prf1} is in general NP-hard, which
establishes the NP-hardness of problem \eqref{eq:prob_hard2}.
The proof is now completed. 
\end{proof}

We would like to stress here that the classical min-max problem of the form \eqref{eq:minmax}, where the objective $f(\x,\y)$ is differentiable, strongly convex in $\x$, and strongly concave in $\y$, is not NP-hard. 
 To see this, one needs only to observe that for any $(\x,\y) \in \X \times \Y$, the direction $(\nabla_x f(\x,\y)^T,-\nabla_y f(\x,\y)^T)^T$ defines a separating hyperplane regarding any stationary solution. In particular, for any given stationary solution $(\x^*,\y^*)\in \X \times \Y$ we have
\begin{eqnarray*}
& & \left(\nabla_x f(\x,\y),-\nabla_y f(\x,\y)\right)^T (\x-\x^*,\y-\y^*) \\
& = &  \left(\nabla_x f(\x,\y)^T (\x-\x^*) - (\nabla_y f(\x,\y)\right)^T (\y-\y^*) \\
& \le & f(\x^*, \y) - f(\x, \y) - [f(\x,\y^*) - f(\x,\y)] \\
& = & f(\x^*, \y) - f(\x,\y^*) \\
&\le& 0, 
\end{eqnarray*}
where the first inequality in the above derivation is due to the gradient inequality of the convex-concave property, and the second inequality is due to the stationary property:
\[
f(\x^*, \y) \le f(\x^*,\y^*) \le f(\x,\y^*),
\]
for all $(\x,\y) \in \X\times \Y$. 
Using the ellipsoid method (see e.g.~\cite{grotschel2012geometric}), this implies that finding a stationary solution can be done in polynomial-time. 
In terms of finding stationary solutions effectively in practice without resorting to the ellipsoid method, there are several works in the literature in which convergence is shown for the above problem to approximate global solutions (i.e., saddle points) in a finite number of iterations for any given precision; see for instance  \cite{mokhtari2019unified, yan2020optimal}.
Therefore, it is the introduction of the linear coupling between the inner and the outer variables in the constraints (of one of the two problems) that {renders} the problem intractable, even under strong assumptions. To illustrate this point more clearly consider the following example.
\begin{example}
Consider the following  simple linearly constrained min-max problem
    \begin{align*}
    \min\limits_{x \in [0,1]} 
    \left(\max\limits_{\stackrel{ y \in [-1,1]}{x-2y = 0}} f(x,y):=x^{2} - \frac{1}{2}xy - 4y^{2} \right).
    \end{align*}
    Notice that the objective $f(x,y)$ is strongly convex in $x$ and strongly concave in $y$. Then, it is easy to see that inner function takes the form 
    $$\phi(x) = \max\limits_{y \in [-1,1]: x-2y=0} x^{2}-\frac{1}{2}xy -4y^{2} = x^{2}-\frac{1}{2}x \frac{1}{2}x -4\frac{1}{4}x^{2} =-\frac{1}{4}x^{2}.$$ Clearly, $\phi(\x)$ is non-convex; in fact it is strongly concave. Therefore we see that while the objective $f(\x,\y)$ is a simple strongly convex strongly concave function, the resulting value function $\phi(\x) = \max_{\y \in \Y, A\x+B\y\leq \bf{c}} f(\x,\y)$ is non-convex. 
\end{example}

Despite the fact that the above results show that the linearly constrained minimax problems are NP-hard in general, it is still desirable to design efficient algorithms for them (for computing certain stationary solutions).
As we have noted in Sec. \ref{sec:related}, standard algorithms for minimax problems do not apply, and  generic methods for bi-level optimization can be  inefficient. 
Our idea is to explore the structure of these problems  from their dual perspective, and identify equivalent problems that are much easier to optimize. Since our focus is mainly given to the inner-level coupled problems, in the next section we develop a duality theory for problem \eqref{eq:minmax-coupled}.

\section{Duality Theory for \eqref{eq:inn_minimax}}\label{sec:problem_form}

To begin with, let us  define the dual problems. 
\begin{definition}[Dual Problems]
Consider the Lagrangian of problem \eqref{eq:inn_minimax}:
\begin{align}\label{eq:L}
L(\x,\y,\lb)=f(\x,\y) - \lb^{T}(A\x+B\y-\cb),
\end{align}
where $\lb \geq {\bf 0}$ is the vector of Lagrangian multipliers.
Then, define the following three dual problems:
\begin{align}
&\min\limits_{\xx} \left(\min\limits_{\lb \geq \bf{0}} \max\limits_{\yy}  L(\x,\y,\lb)\right) \tag{D1},\label{eq:D1}\\
&\min\limits_{\lb \geq \bf{0}} \left(\min\limits_{\xx} \max\limits_{\yy}  L(\x,\y,\lb)\right) \tag{D2},\label{eq:D2}\\
&\min\limits_{\lb \geq \bf{0}, \xx} \left( \max\limits_{\yy}  L(\x,\y,\lb)\right) \tag{D3}.\label{eq:D3}
\end{align} 
\end{definition}

Note that due to the special structure of our problem, we can potentially have a few more possible  `dual' problems.  However, our analysis will  focus on three of them listed above, since they will be instrumental in our subsequent analysis and development of stationary solution concept. Below, we provide the weak duality theorem of problem \eqref{eq:inn_minimax}. 
\begin{theorem}[Weak duality]\label{pro:weak_dual}
Under Assumption \ref{as:feas}, we have: 
\begin{align*}
v \left( \mbox{\rm \ref{eq:inn_minimax}} \right) \leq 
v \left( {\rm \ref{eq:D1}} \right) = 
v \left( {\rm \ref{eq:D2}} \right)= 
v \left( {\rm \ref{eq:D3}} \right).
\end{align*}
\end{theorem}
\begin{proof}
First, notice that we can exchange the order of the min operators:
$$\hspace{-1mm}
\min\limits_{\xx} \left(\min\limits_{\lb \geq {\bf{0}}} \max\limits_{\yy}  L(\x,\y,\lb) \right) = \min\limits_{\lb \geq {\bf{0}}} \left(\min\limits_{\xx} \max\limits_{\yy}  L(\x,\y,\lb \right)
=\min\limits_{\lb \geq {\bf{0}}, \xx} \left( \max\limits_{\yy}  L(\x,\y,\lb)\right).\hspace{-2mm}$$ So
$v \left({\rm \ref{eq:D1}} \right) = v \left({\rm \ref{eq:D2}} \right) = v \left( {\rm \ref{eq:D3}} \right)$.
Moreover, using the max-min inequality we have that
\begin{align*}
\max\limits_{\yy} \min\limits_{\lb \geq {\bf{0}}} L(\x,\y,\lb) 
&\leq 
\min\limits_{\lb \geq {\bf{0}}} \max\limits_{\yy}  L(\x,\y,\lb), \; \forall \x \in \X.
\end{align*} 
Minimizing the problems on both sides over the set $\mathcal{X}$, we can obtain 
\begin{align}\label{eq:prf_weak1}
\min\limits_{\xx} \left(\max\limits_{\yy} \min\limits_{\lb \geq {\bf{0}}} L(\x,\y,\lb) \right)
&\leq 
\min\limits_{\xx} \left(\min\limits_{\lb \geq {\bf{0}}} \max\limits_{\yy}  L(\x,\y,\lb) \right).
\end{align} 
Moreover, consider an arbitrary $\xx$, and observe that if there exists an index $i$ such that $\left[A\x + B\y -\cb \right]_{i} >0$, then we have:
\begin{align}\label{eq:prf_weak2}
\min\limits_{\lb \geq {\bf{0}}} L(\x,\y,\lb) = f(\x,\y) + \min\limits_{\lb \geq {\bf{0}}} \left\{ -\lb^{T}\left(A\x + B\y -\cb \right) \right\} = - \infty.
\end{align}
On the other hand, if it holds that $A\x + B\y \leq \cb $, then:
\begin{align}\label{eq:prf_weak3}
\min\limits_{\lb \geq {\bf{0}}} L(\x,\y,\lb) = f(\x,\y) > \underline{f}, \quad \forall\; \x\in \mathcal{X}, \y\in \mathcal{Y}.
\end{align}
As a result, \eqref{eq:prf_weak2} and \eqref{eq:prf_weak3} imply that: 
\begin{align*}
\max\limits_{\yy} \left(\min\limits_{\lb \geq {\bf{0}}} L(\x,\y,\lb) \right) &=
\max\limits_{\yy,  A\x+B\y \leq {\bf{c}} } \left(\min\limits_{\lb \geq {\bf{0}} } L(\x,\y,\lb) \right)=\max\limits_{\yy,  A\x+B\y \leq {\bf{c}} }  f(\x,\y). 
\end{align*}
Minimizing both sides of the above equality over $\X$ we get: 
\begin{align}\label{eq:prf_weak4}
\min\limits_{\x \in \X} 
\left( \max\limits_{\yy, A\x+B\y \leq {\bf{c}}} f(\x,\y) \right)
= \min\limits_{\xx} \left(\max\limits_{\yy} \min\limits_{\lb \geq {\bf{0}}} L(\x,\y,\lb) \right).
\end{align} 
Finally, combining \eqref{eq:prf_weak1} and \eqref{eq:prf_weak4}, we conclude that:
$$
\min\limits_{\x \in \X} 
\left( \max\limits_{\yy, A\x+B\y \leq {\bf{c}}} f(\x,\y) \right)
\leq \min\limits_{\xx} \left(\min\limits_{\lb \geq {\bf{0}}} \max\limits_{\yy}  L(\x,\y,\lb) \right).
$$
The proof is now completed.
\end{proof}
 
Next, we develop the strong duality of problem \eqref{eq:inn_minimax}. That is, we identify the conditions under which strong duality holds, and establish the equivalence of the solutions of the primal problem with those of the dual problems. To begin with, we impose the following assumptions:
\begin{assumption}\label{as:strong_dual}
Assume that the following holds:
\begin{enumerate}
\item {$f(\x,\y)$ is strongly concave in $\y$,  for every $\x \in \X$, with modulus $\mu_y$. \label{ass:str_cnc_y}} 

\item  For every $\xx$ there exists  $\y \in {\rm relint}(\Y)$ such that $(A\x + B\y - \cb) \leq {\bf 0}$. \label{ass:slater}

\end{enumerate}
\end{assumption}
Based on the above assumptions, we have the following duality theorem.
\begin{theorem}[Duality Theorem]\label{pro:strong_dual}
Under Assumptions \ref{as:feas} and \ref{as:strong_dual}, strong duality holds, that is, we have the following relations:
\begin{align*}
v \left( \mbox{\rm \ref{eq:inn_minimax}} \right)= 
v \left( {\rm \ref{eq:D1}} \right)= 
v \left( {\rm \ref{eq:D2}} \right)= 
v \left( {\rm \ref{eq:D3}} \right).
\end{align*}

Also, for the solutions of the dual problems we have that: 

\begin{enumerate}

\item  There exists a $\lb^{\ast} \geq {\bf 0}$ such that $(\x^*,\y^*,\lb^*)$ is a solution of \eqref{eq:D2} if an only if $(\x^*,\y^*)$ is a solution of \eqref{eq:inn_minimax}.   

That is, the tuple $(\x^*,\y^*,\lb^*)$ satisfies the following condition: {\it 1)} $(\x^*,\y^*)$ is a solution (i.e., a minimax point) of the inner min-max problem with $\lb=\lb^{\ast}$, i.e., $
       (\x^*,\y^*) \in \arg\min\limits_{\xx} \arg\max\limits_{\yy}  L(\x,\y,\lb^{\ast});$ {\it 2)} $\lb^{\ast}$ is a global minimizer of the function $G(\lb):= \min\limits_{\xx} \max\limits_{\yy}  L(\x,\y,\lb)$, $\lb^{\ast} \in \arg\min\limits_{\lb \geq {\bf 0}} G(\lb).$

\item  There exists a $\lb^{\ast} \geq {\bf 0}$ such that $(\x^*,\y^*,\lb^*)$ is a solution of \eqref{eq:D3} if an only if $(\x^*,\y^*)$ is a solution of \eqref{eq:inn_minimax}.  

That is:
      $\left((\x^*,\lb^*),\y^*\right) \in  \arg\min\limits_{\lb \geq \bf{0}, \xx}  \arg\max\limits_{\yy}  L(\x,\y,\lb).$
\end{enumerate}

\end{theorem}
\begin{proof}

We divide the proof into two parts. 

\noindent {\bf Objective equivalence.} We already established in Theorem \ref{pro:weak_dual} that 
$v \left( {\rm \ref{eq:D1}} \right) = v \left( {\rm \ref{eq:D2}} \right)= v \left( {\rm \ref{eq:D3}} \right)$.
  Note that strong duality holds for problem $\max\limits_{\y \in \Y, A\x+B\y\leq\bf{c}} f(\x,\y)$, for any fixed $\x$, as a consequence of Assumption  \ref{as:strong_dual} (see, e.g., \cite[Prop. 5.3.1]{bertsekas2009convex}). Then,
  {\small
\begin{align*}
&\max\limits_{\y \in \Y, A\x+B\y\leq\bf{c}} f(\x,\y) 
= \max\limits_{\yy} \min\limits_{\lb \geq \bf{0}} L(\x,\y,\lb) 
=\min\limits_{\lb \geq \bf{0}} \max\limits_{\yy}  L(\x,\y,\lb), \; \forall \xx
\\
&\min\limits_{\xx} \left( \max\limits_{\y \in \Y, A\x+B\y\leq\bf{c}} f(\x,\y)\right)
=\min\limits_{\xx} \left(\max\limits_{\yy} \min\limits_{\lb \geq \bf{0}} L(\x,\y,\lb) \right)
=\min\limits_{\xx} \left(\min\limits_{\lb \geq \bf{0}} \max\limits_{\yy}  L(\x,\y,\lb) \right).
\end{align*}
}
That is, we have $v \left( \mbox{\rm \ref{eq:inn_minimax}} \right) = v \left( {\rm \ref{eq:D1}} \right) = v \left( {\rm \ref{eq:D2}} \right)= v \left( {\rm \ref{eq:D3}} \right)$.

\noindent{\bf Solution equivalence.} Next we  show the second part of the claim. 

First, we show that the solutions of \eqref{eq:inn_minimax} and \eqref{eq:D1} are equivalent.   Consider an arbitrary $\xx$, and let us define the following problems  
\begin{align*} 
&(A) : \quad g_{A}(\x):= \max\limits_{\yy} \min\limits_{\lb \geq 0} L(\x,\y,\lb)\\
&(B) : \quad g_{B}(\x):= \min\limits_{\lb \geq 0} \max\limits_{\yy}  L(\x,\y,\lb)  \\
&(C) : \quad g_{C}(\x):= \max\limits_{\y \in \mathcal{Y},A\x+B\y \leq\cb} f(\x,\y).
\end{align*}

Let us denote  $\left(\y^*_A,\lb^*_A\right)$ as a solution of (A), and with $(\y^*_B,\lb^*_B)$ a solution of (B). We will show below that if $(\y^*_B,\lb^*_B)$ is a solution of (B), then $\y^*_B$ is a solution of (C). Conversely, we will show that if $\y^*$ is a solution of (C), then there exists a $\lb^*$ such that $(\y^*,\lb^*)$ is a solution of (B). 

To begin with,  
 using  \cite[Prop. 5.3.1]{bertsekas2009convex}, we see that Assumption  \ref{as:strong_dual} implies that strong duality holds for problem (C).
From \cite[Prop. 5.3.2]{bertsekas2009convex} we know that when strong duality holds, $(\y^*_A, \lb^*_B)$  satisfies the following conditions: $\lb^*_B \geq \bf{0}$, $\y^*_A \in \Y$ and $A\x + B\y^*_A - \cb \leq \bf{0}$, and 
\begin{align}\label{eq:dual_prf_1}
\y^*_A = \arg\max\limits_{\yy} L(\x,\y,\lb^*_B).
\end{align}
Moreover, observe that $(\y^*_B,\lb^*_B)$ is a solution of (B), and thus
\begin{align}\label{eq:dual_prf_2}
\y^*_B = 
\arg\max\limits_{\yy} L(\x,\y,\lb^*_B) =
\arg\max\limits_{\yy} \left\{ f(\x,\y) - (\lb^*_{B})^{T}(A\x+B\y-\cb)\right\}.
\end{align}
Then, combining  \eqref{eq:dual_prf_1} and \eqref{eq:dual_prf_2}, and considering the strong concavity of $L(\x,\y,\lb^*_B)$ w.r.t $\y$, we can infer that $\y^*_A=\y^*_B$. Therefore, $\y^*_B$ is a solution of problem (A), and as a result a solution of (C). 
 
Conversely, we will show that if $\y^*$ is a solution of (C) then there exists a $\lb^*$ such that $(\y^*,\lb^*)$ is a solution of (B). Indeed, under strong duality, for the unique (due to strong concavity) solution $\y^*$ of (C) there exist Lagrange multipliers $\lb^* \geq {\bf 0}$ such that $(\y^*, \lb^*)$ is a saddle point of the Lagrangian $L(\x,\y,\lb)$ \cite[Prop. 5.3.2]{bertsekas2009convex}. We know that the set of saddle points of $L(\x,\y,\lb)$ is the intersection of the set of minimax points $(\y^*_A, \lb^*_A)$ with the set of maximin points $(\y^*_B, \lb^*_B)$, and thus it is implied that $(\lb^*, \y^*)$ is a solution of (B). 

We have shown for some arbitrary $\xx$ that  $\y^*_{\x}$ is a solution of (C) if and only $(\y^*_{\x}, \lb^*_{\x})$ is a solution of (B), for some $\lb^*_{\x} \geq {\bf{0}}$; here we used the notation $\y^*_{\x}, \lb^*_{\x}$ to make the dependence to that arbitrary $\x$ clear.
Finally, note that the strong duality holds for problem (C)  for any fixed $\xx$. This implies that the value of the problems (A), (B), and (C) is the same at every $\xx$, that is 
$$g_{C}(\x) = g_{A}(\x) = g_{B}(\x), \; \forall \xx.$$ 
Therefore, $g_{C}, g_{A}, g_{B}$ have the same set of global minima $\X^*$. Then, at any given $\x^* \in \X^*$, $\y^*_{\x^*}$ is a solution of (C) if and only if $(\y^*_{\x^*}, \lb^*_{\x^*})$ is a solution of (B), for some $\lb^*_{\x^*} \geq {\bf{0}}$. 
In conclusion, a point $(\x^*, \y^*_{\x^*})$ is a solution of \eqref{eq:minmax-coupled} if and only if $(\x^*, \y^*_{\x^*}, \lb^*_{\x^*})$ is a solution of \eqref{eq:D1}.
The claim is proved.

Next, we show that the solutions of \eqref{eq:inn_minimax} and \eqref{eq:D2}/\eqref{eq:D3} are equivalent. Note that the solution sets of the problems \eqref{eq:D1}, \eqref{eq:D2} and \eqref{eq:D3} are equivalent, since the order of the min operators can be exchanged. That is, a solution $(\x^*,\y^*,\lb^*)$ of \eqref{eq:D2} is also a solution of \eqref{eq:D1}, and vice versa; the similar result holds true between \eqref{eq:D3} and \eqref{eq:D1}.
In addition, using the (solution) equivalence between problems \eqref{eq:inn_minimax} and \eqref{eq:D1} established above, we can conclude that 
a point $(\x^*,\y^*,\lb^*)$ is a solution of \eqref{eq:D2} if an only if $(\x^*,\y^*)$ is a solution of \eqref{eq:inn_minimax}. The same result also holds for \eqref{eq:D3}.
The proof is now completed.
\end{proof}

 In the above proof, note that despite Assumption \ref{as:strong_dual}, which ensures strong duality, the solution sets of the problems (A) and (B) (denoted with $S_{A}, S_{B}$, respectively) are not the same. Nonetheless, the original problem \eqref{eq:minmax-coupled} and the dual one \eqref{eq:D2} are equivalent (in the sense mentioned in Theorem \ref{pro:strong_dual}). This is due to the fact that $S_B \subseteq S_A$, which implies that every solution of $B$ is a saddle point (i.e., the points in the set $S_{saddle}= S_B \cap S_A$) of the Lagrangian $L$ and vice versa.
Below we provide a relevant example.

\begin{example}\label{ex:count_ex}
Consider 
 a problem of the form \eqref{eq:minmax-coupled}, where the inner task takes the form
$\max_{y \in [0,2], y \leq 1} -y^{2}+2y$;
 the exact form of the outer task is irrelevant for the context of this example.
Then, the primal (A) and dual (B) problems  (of the inner task) are  
$$p^{\ast}=\max_{y \in [0,2]} \min_{\lambda \geq 0} \left\{-y^{2}+2y - \lambda (y-1)\right\}$$
$$d^{\ast} = \min_{\lambda \geq 0} \max_{y \in [0,2]} \left\{-y^{2}+2y - \lambda (y-1)\right\}.$$ 
 Note that Assumption \ref{as:strong_dual} is satisfied in this example, and thus $p^{\ast}=d^{\ast}$ (i.e., strong duality). However, it can be easily shown that the solutions set of the primal problem is $S_A = \{(1,\lambda): \lambda \geq 0 \}$, while the respective set of the dual problem is different and consists only of one point, namely $S_B = \{(1,0)\} \subseteq S_{A}$. As a result, there is a unique saddle point, i.e., $S_{saddle} = S_B \cap S_A = S_B = \{(1,0)\}$.
\end{example}

\section{First-order stationarity conditions}\label{sec:appr}

In Sec. \ref{sec:coupled} we established that finding the globally optimal solutions of problem \eqref{eq:inn_minimax} is NP-hard in general. It is then useful to identify some high-quality solution concepts that can be computed efficiently.  In this section, we propose to leverage the duality theory developed in the previous section to identify such a class of first-order stationary solutions. More specifically, we will define first-order stationary solutions based on the  dual problem \eqref{eq:D2}. Note that one can directly work with the problem $\phi(\x) = \max_{\y \in \Y, A\x+B\y\leq\bf{c}} f(\x,\y)$, but this is challenging since $\phi(\x)$ is neither smooth nor convex. 

Before we proceed let us make a few additional  assumptions:
\begin{assumption}\label{as:f_str_cvx_x}
We impose the following assumptions
\begin{enumerate}

\item \label{ass:diff} The function $f(\x,\y):\R^{n} \times \R^{m} \to \R$ is differentiable w.r.t. both $\x$ and $\y$.

\item \label{ass:f_str_cvx_x} The function $f(\x,\y)$ is strongly convex in $\x$, for every $\y \in \Y$, with modulus $\mu_x$.

\item \label{ass:Lipschitz} The function $f$ has Lipschitz continuous gradients, i.e.,
\begin{align*}
&\left\|\gx f(\x_1,\y_1) - \gx f(\x_2,\y_2)\right\| \leq L_{x} \left\|\begin{bmatrix} \x_1 \\  \y_1 \end{bmatrix} - \begin{bmatrix} \x_2 \\  \y_2 \end{bmatrix} \right\|, \forall \x_1,\x_{2} \in \X, \y_1,\y_2 \in \Y.\\
&\left\|\gy f(\x_1,\y_1) - \gy f(\x_2,\y_2)\right\| \leq L_{y} \left\|\begin{bmatrix} \x_1 \\ \y_1 \end{bmatrix} - \begin{bmatrix} \x_2 \\  \y_2 \end{bmatrix} \right\|,
\forall \x_1,\x_{2} \in \X, \y_1,\y_2 \in \Y.
\end{align*}
\end{enumerate}
\end{assumption}


 Next, let us define the following quantities:

\begin{subequations}
\begin{align}
   H(\x,\lb) &:= \max_{\y \in \Y} L(\x,\y,\lb) \label{eq:H}\\
    G(\lb) & := \min\limits_{\xx}  \max\limits_{\yy}  L(\x,\y,\lb) \label{eq:def:G}\\
    \yo(\x,\lb) & := \arg\max\limits_{\yy} L(\x,\y,\lb), \;\;
    \xo(\lb) := \arg\min\limits_{\xx}  H(\x,\lb). \label{eq:argminmax}
    \end{align}
\end{subequations}
Using these new notations, problem \eqref{eq:D2} can be decomposed into the following two problems:
\begin{align}
  \min\limits_{\xx} H(\x,\lb) & : = \min_{\x\in \mathcal{X}}\left(\max_{\y\in \mathcal{Y}}\; L(\x,\y,\lb)\right)  
  \quad  \mbox{(Inner-Level Problem)} \label{eq:D2_inn}\\
  \min\limits_{\lb \geq \bf{0}} G(\lb) &:= \min\limits_{\lb \geq \bf{0}} \left( \min\limits_{\xx} H(\x,\lb) \right)\quad \quad  \; \mbox{(Outer-Level Problem)}. \label{eq:D2_out}
\end{align}

Note that the assumption that $f$ is strongly-convex in $\x$ provides us with a number of useful properties which facilitate algorithm design. 
Specifically, it implies that the function $G(\lb)$ defined above, which is the objective of the dual problem \eqref{eq:D2}, is differentiable. However, notice that $G(\lb)$ is not necessarily convex and therefore any reasonable notion of stationary conditions for (D2) will involve the gradient and the stationary points of $G(\lb)$. Therefore, below we derive the formula of the gradient of $G(\lb)$.
\begin{lemma}\label{lem:inminmax_grad}
Suppose that Assumptions \ref{as:feas}, \ref{as:strong_dual}, and \ref{as:f_str_cvx_x} hold. Then $G(\lb)$ is differentiable, and its gradient is given by:
$$\g G(\lb) = -A\xo\left(\lb\right) - B\yo \left(\xo\left(\lb\right),\lb \right) + \cb,$$ 
where $\xo\left(\lb\right),\yo\left(\x,\lb\right)$ are defined in  \eqref{eq:argminmax}.
\end{lemma}
\begin{proof}
See Appendix \ref{app:solution}.
\end{proof}
 Now we are ready to provide definitions of (exact and approximate) stationary solutions, based on the dual problem \eqref{eq:D2}.

\begin{definition}[Stationary Solutions]\label{def:D2_stat}
A solution $\left(\xo,\yo,\lo\right)$ is an exact stationary solution of \eqref{eq:D2} if the following holds:
\begin{itemize}
    \item $(\xo,\yo)= \left( \xo\left(\lo\right),\yo\left(\xo,\lo\right) \right)$;
    \item $\lo$ is a stationary point of $G(\lb)$, that is
    $\left\langle \g G\left(\lo\right), \lb-\lo\right\rangle \geq 0, \quad \forall \lb \geq 0$.
\end{itemize}
\end{definition}
We will later show that the stationarity conditions for $G(\lb)$ imply the complementarity condition $\lo^{T} \g G\left(\lo\right)  =0$ with $\lb \geq {\bf 0}$. Next, we provide the definition of the approximate stationary solutions.
\begin{definition}[$(\epsilon, \delta)$-approximate Stationary Solution]\label{def:D2_approx_stat}
Let us define the following quantities:
\begin{align*}
Q(\lb) &:= \frac{1}{\alpha} \left(\lb - \text{\rm proj}_{\R^{k}_{+}} \left(\lb - \alpha \g G(\lb) \right) \right),\\
d(\x,\y, \lb) & := \|\xo(\lb) - \x \|^2 + \|\yo(\xo(\lb),\lb) - \y \|^2,
\end{align*}
where $\alpha>0$ is a constant independent of $\epsilon$ and $\delta$.
Then, a point $\left(\xt,\yt, \lt\right) \in \X\times\Y\times\R^{k}_{+}$ is called an $(\epsilon, \delta)$-approximate stationary solution if it holds that 
\begin{align*}
\left\| Q\left(\lt\right) \right\| \leq \epsilon, \;\;
d\left(\xt,\yt, \lt\right) \leq \delta^{2}.
\end{align*}
\end{definition}

In words, at  an exact stationary solution of \eqref{eq:D2}, the inner-level minimax problem is solved to global optimality, while the outer-level one reaches  an exact stationary solution. 
However, finding such points in practice is unrealistic. The notion of approximate stationary solutions is then introduced as a relaxation. At an approximate stationary solution $\left(\xt,\yt, \lt\right)$ the inner-level problem is solved inexactly, in the sense that the distance of $(\xt,\yt)$ from the exact solution (of the inner-level problem \eqref{eq:D2_inn} with $\lb=\lt$) is small. In addition, the function $Q(\lt)$ is the stationarity gap of the outer-level problem at $\lb=\lt$. 

Next, we analyze the implications of a point being an (exact or approximate) stationary solution. In particular, 
we will show that at an exact stationary solution, the coupled linear constraints  as well as the complementary slackness condition will be satisfied.  
Further, at an approximate stationary solution, the choice of $\lt$ ensures that  the constraint violation is small.

\begin{proposition}\label{pro:stat} 
Suppose that Assumptions \ref{as:feas}, \ref{as:strong_dual}, \ref{as:f_str_cvx_x} hold, and let $\left(\xo,\yo,\lo\right)$ be a stationary solution of \eqref{eq:D2}. Then the following holds
\begin{subequations}
\begin{align}
 \yo  & = \arg\max_{\yy} L(\xo,\y,\lo),
 \quad A\xo + B\yo - \cb \leq {\bf{0}}, \quad \xo  =\arg\min_{\xx} \left\{H\left(\x,\lo\right) \right\} \label{eq:optimality} \\
 0& = \lo^{T} \left(-A\xo\left(\lo\right) - B\yo\left(\xo\left(\lo\right),\lo\right) + \cb\right), \quad  \lb \geq {\bf 0} . \label{eq:complementarity}
\end{align}
\end{subequations}

\end{proposition}
\begin{proof}
See Appendix \ref{app:solution}. 
\end{proof}

\begin{proposition}\label{pro:appr_stat}
Suppose that Assumptions \ref{as:feas}, \ref{as:strong_dual}, \ref{as:f_str_cvx_x} hold, and let us denote with $\left(\xt,\yt, \lt\right) \in \X\times\Y\times\R^{k}_{+}$ an $(\epsilon, \delta)$-approximate stationary solution. Then, $ \forall i \in \mathcal{K}$  it holds that:
\begin{align*}
&\max\left\{0,\left[A\xt+B\yt-\cb \right]_i\right\} \leq 2 \sigma_{\rm max}  \cdot  \delta + \epsilon, \\
&  \|\yo - \yt \| \leq \delta, \; \|\xo - \xt \|\leq \delta,
\end{align*}
 where $(\xo,\yo) = \left( \xo\left(\lt\right),\yo\left(\xo,\lt\right) \right)$, and the two functions $ \xo(\cdot),\yo(\cdot)$ are defined in \eqref{eq:argminmax}; $\sigma_{\rm max}:=\max\left\{\|A\|,\|B\|\right\}$. 
\end{proposition}
\begin{proof}
See Appendix \ref{app:solution}. 
\end{proof}

So far, we have defined a suitable (dual) reformulation of the original problem \eqref{eq:inn_minimax}, and introduced a set of first-order stationarity conditions for it. To evaluate the quality of these solutions, we introduce an algorithm named Multiplier Gradient Descent (MGD) that can efficiently compute such solutions. Next, we provide a description of the proposed  algorithm.

First, we obtain an approximate solution of the inner-level problem \eqref{eq:D2_inn}, by using any reasonable  iterative subroutine that can solve a strongly-convex strongly-concave minimax problem. In Algorithm \ref{alg:alg}, such a subroutine is referred to as `Alg$(\cdot)$', which takes the current iterates $\x^{r},\y^{r}, \lb^{r}$ as well as the desired number of inner iterations as input. 
Such an algorithm can be, for instance, the gradient descent-(multi-step) ascent \cite{nouiehed2019solving} or the extragradient method \cite{mokhtari2019unified}. 
Afterwards, one iteration of gradient descent is performed on the outer problem \eqref{eq:D2_out}.  
In the subsequent discussion, we refer to $r$ as the {\it outer-iteration} index.

\begin{algorithm}[tb]
   \caption{Multiplier Gradient Descent (MGD)}
\begin{algorithmic}
   \STATE {\bfseries Input:} $\x^0,\y^0, \lb^0, T, K, \alpha, A, B, \cb$
   \FOR{$r=0$ {\bfseries to} $T-1$}
   \STATE \# Solve inner minimax problem by running Alg using $K$ iterations;
   \STATE \quad \# e.g. Alg: Gradient Descent-(multi-step) Ascent (GDA) ;
   \STATE \quad  $\left(\x^{r+1},\y^{r+1}\right) \leftarrow {\rm Alg}\left(\x^{r},\y^{r}, \lb^{r};K\right)$ 
   \STATE \# Solve outer min problem
   \STATE $\lb^{r+1} = \text{proj}_{\R^{k}_{+}} \left[\lb^{r} - \alpha \left( -A\x^{r+1} - B\y^{r+1} +\cb \right)\right]$
   \ENDFOR
\end{algorithmic}
\label{alg:alg}
\end{algorithm}

 
We should note here that the MGD algorithm resembles the classical dual-ascent algorithm (for constrained minimization problems), and its analysis also follows similar ideas. The  interested readers are referred to 
the online version of this paper for the formal convergence claims and the proof \cite{tsaknakis2021minimax}. We emphasize that algorithm design and analysis are {\it not} the central focus of this work. The MGD algorithm  just gives us a way to evaluate the quality of the proposed stationary solution. One may certainly be able to design more efficient algorithms (for example, based on extra-gradient, or momentum acceleration techniques), but this is beyond the scope of this work.  

Below we briefly comment on the convergence properties of MGD.

\begin{remark}
First of all, under the Assumptions \ref{as:feas}, \ref{as:strong_dual}, and~\ref{as:f_str_cvx_x} we can show that the outer-level objective $G(\lb)$ has Lipschitz continuous gradients. Based on this result, 
 we can show (under the same assumptions as above) that the MGD  algorithm can asymptotically compute an \textit{exact stationary solution}, i.e., 
$$\lim_{r\to\infty} \|Q(\lb^r)\| = 0, \quad \lim_{r\to\infty} d(\x^{r},\y^{r},\lb^r) =0.$$ 
This result implies that the properties of Proposition \ref{pro:stat} are also satisfied asymptotically, in particular, the complementary slackness condition holds, i.e.,
$$\quad \lim_{r\to\infty} \left\langle \lb^{r}, A\xo\left(\lb^{r}\right) + B\yo\left(\xo\left(\lb^{r}\right),\lb^{r}\right) - \cb  \right\rangle = 0. $$

Moreover, we can prove that MGD can reach an $(\epsilon,\delta)$-\textit{approximate stationary solution} using at most $\mathcal{O}( \frac{1}{\epsilon^2} +  \frac{1}{\sqrt{\delta}})$ outer iterations. 
However, in this case in order to ensure that the complementary slackness violation also goes to zero (as a function of $\epsilon$ and $\delta$), we need to impose the following regularity assumption: for every $\x\in \mathcal{X}$, $0$ is in the interior of the set $\{A \x + B\y  - \cb \mid \y\in \mathcal{Y}\}$; note that this assumption is only slightly more restrictive than the one made in  Assumption \ref{as:strong_dual}.\ref{ass:slater}. Then, we can show that 
at an $(\epsilon, \delta)$-approximate stationary solution $\left(\xt,\yt, \lt\right)$ obtained by the MGD algorithm, the complementary slackness violation $\langle \tilde{\lb},  A\xt + B\yt - \cb\rangle$ is upper and lower bounded by functions of $\epsilon$ and $\delta$, and these functions converge to zero as  $\epsilon\to 0, \delta\to 0$.
\end{remark}

\section{Experiments}\label{sec:exp}

In this section we perform a number of experiments that develop attacks for a minimum cost network flow problem,  using the proposed formulation \eqref{eq:net_atttack}. Note that our main goal is to evaluate the quality of the formulation and the proposed stationary solutions. To begin with, in our experiments we adopt the setting described in Sec. \ref{sec:net_flow}, and  generate networks of $n$ nodes at random using the Erdos-Renyi model with parameter $p$ (i.e., the probability that an edge appears on the graph is $p$).
Moreover, the capacity $p_{ij} $ and the cost coefficients $w_{ij}$ of the edges are generated uniformly at random in the interval $[1,2]$, while we set the demand (on the sink) as $d\%$ of the sum of capacities of the edges exiting the source, and $d$ is a parameter to be chosen.

The adversary will generate the attack by solving problem \eqref{eq:net_atttack}. We choose $q_{e}(x_{e}) = w_{e}x_{e}$, that is the  cost per unit flow is a function of the amount of flow.
Further, we add a small regularizer to the adversarial's problem to make it strongly concave. The specific problem to be solved is listed below ($\eta > 0$ is a small constant): 
\begin{align}\label{eq:experiment}
	\max\limits_{\stackrel{\bf{0} \leq \bf{y} \leq \bf{p}}{\sum_{(i,j) \in E} y_{ij} = b}} 
	 & \min \limits_{\stackrel{\bf{0} \leq \bf{x} \leq \bf{p}}{\sum\limits_{(i,t) \in E} x_{it} = r_t} } \sum_{(i,j) \in E} w_{ij}\cdot (x_{ij}+y_{ij})\cdot x_{ij} -  \frac{\eta}{2} \|\y\|^2\\
	 \text{s.t. } &  \bf{x} + \bf{y} \leq \bf{p} \nonumber \\
	 &\hspace{-3mm}  \sum_{(i,j) \in E} x_{ij} - \sum\limits_{(j,k) \in E} x_{jk}= 0, \; \forall j \in V \setminus \{s,t\} \nonumber.
\end{align}

The above problem is then solved using the following two different approaches, while the algorithm parameters are chosen so that the performances are optimized:
\begin{enumerate}
    \item Apply the MGD algorithm (Algorithm \ref{alg:alg}) to solve problem \eqref{eq:D2}.
    Choose $T=100$, and use $K=5$ steps of the gradient descent-ascent method \cite{lin2020gradient} to solve the inner-problem. The stepsizes for all variables are set to $0.5$.
    \item Apply the gradient descent-ascent algorithm (with multiple descent steps) \cite{nouiehed2019solving} to solve problem \eqref{eq:D3}. Choose $T=100$, and run $5$ iterations for the inner minimization problem. The stepsizes for all variables are set to $0.5$. We refer to this approach as GDA.
\end{enumerate}
To evaluate the performance of the proposed attack we test the following baselines:
\begin{itemize}
\item \underline{Random attack}: An attack flow vector $\x$ is generated at random, under the given budget constraints $b$.

\item \underline{Max capacity attack}: The attack removes the edge with the largest capacity.

\item \underline{Greedy attack}: The edges are sorted in an ascending order w.r.t. the cost coefficients, and they are removed following this order until the capacity budget $b$ is met. If the remaining budget does not suffice to remove an edge (i.e., zero its capacity), then the capacity of the edge is reduced by this amount.

\item  \underline{NI attack}: We generate an attack following a method that is used in (Generalized) Nash Equilibria games. Specifically, we minimize a Nikaido-Isoda (NI)-based reformulation \cite{facchinei2010generalized, von2009optimization} of our problem, that is
\begin{subequations}
\begin{align}
     V(\x,\y) &= -V_{1}(\y) - V_{2}(\x)\label{eq:v}\\
    {\rm where}\;\;   V_1(\y) & : =\min_{\x \in \X} f(\x,\y), \quad V_2(\x): =\max_{\y \in \Y} f(\x,\y) \label{eq:v1v2}.
\end{align} 
\end{subequations}
We approximately solve the (strongly-convex or strongly-concave) problems in \eqref{eq:v1v2} using 25 steps of gradient descent and gradient ascent, respectively. Then, the objective  $V$ is minimized by using $100$ steps of gradient descent. 

\end{itemize}

The performance of the experiments are evaluated according to the following procedure. First, an attack is generated, and the corresponding link capacities are reduced according to the attack pattern.   Then, the minimum cost flow is computed based on the available link capacities. Let $\x_{\rm att}, \x_{\rm cl}$ denote the minimum cost flow assignment after and before the attack, respectively.   We define the following {\it relative} increase of the cost as a performance measure:
$\rho: = \frac{q_{\rm tot}(\x_{\rm att}) - q_{\rm tot}(\x_{\rm cl})}{q_{\rm tot}(\x_{\rm cl})}.$
The higher the increase the more successful/powerful the attack. Furthermore, in our experiments we compute the above measure as a function of the adversary's budget, and for each budget level we perform the experiment $15$ times (each time a new graph is created, with different capacity and cost vectors) and average the results. 

The results are illustrated in Fig. \ref{fig:min_cost_exp}. Observe that the proposed minimax based attack is more powerful compared  to the three baselines, as it leads to the largest minimum cost flow after the attack in both networks (regardless of which algorithm we use).  This also implies that the stationary solutions defined in the previous section correspond to relatively ``strong'' attacks. We also notice that MGD is slightly more advantageous compared to the GDA. It is worth mentioning that the network we generated makes the attack problem {\it non-trivial}, in the sense that random attack does not work (the relative cost increase is imperceptible).  
Further, we would like to point out that the relatively high variance exhibited in the experiments is not unexpected since at every run we generate a different graph. In particular, in the experiments with $p=0.75$, every new graph differs significantly from the previous one, leading to higher variance in this case.

\begin{figure}
\centering
\begin{subfigure}[b]{0.35\textwidth}
\centering
\includegraphics[scale=0.35]{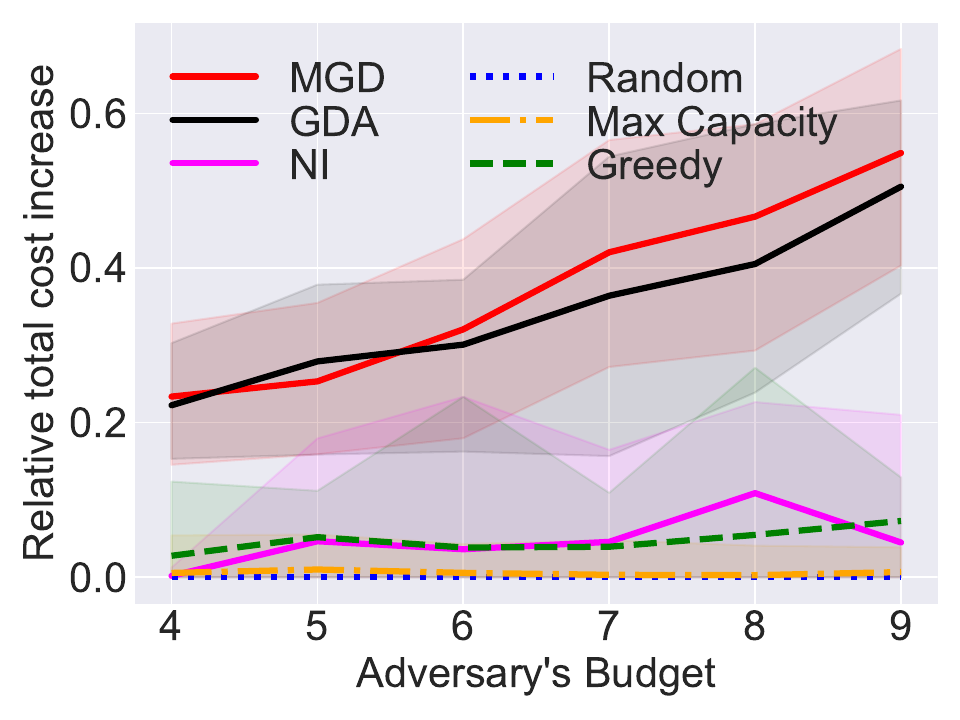}
\caption{$n=15$, $p=1$, $d=20\%$}\label{fig:cost_budget11}
\end{subfigure}
\hspace{4mm}
\begin{subfigure}[b]{0.35\textwidth}
\centering
\includegraphics[scale=0.35]{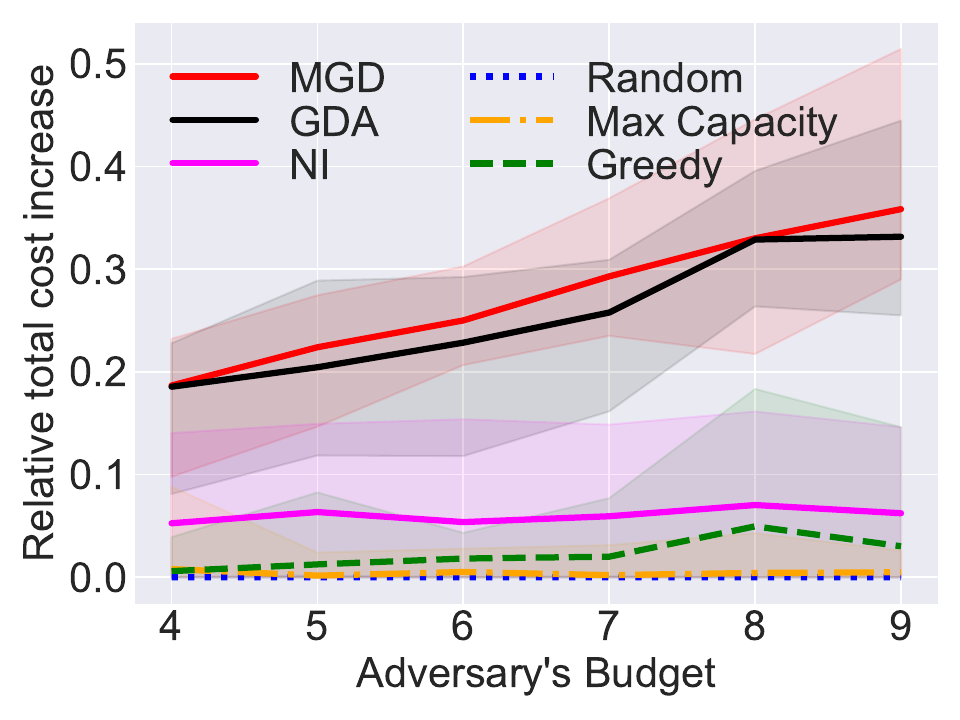}
\caption{$n=20$, $p=1$, $d=30\%$}\label{fig:cost_budget12}
\end{subfigure}
\\
\begin{subfigure}[b]{0.35\textwidth}
\centering
\includegraphics[scale=0.35]{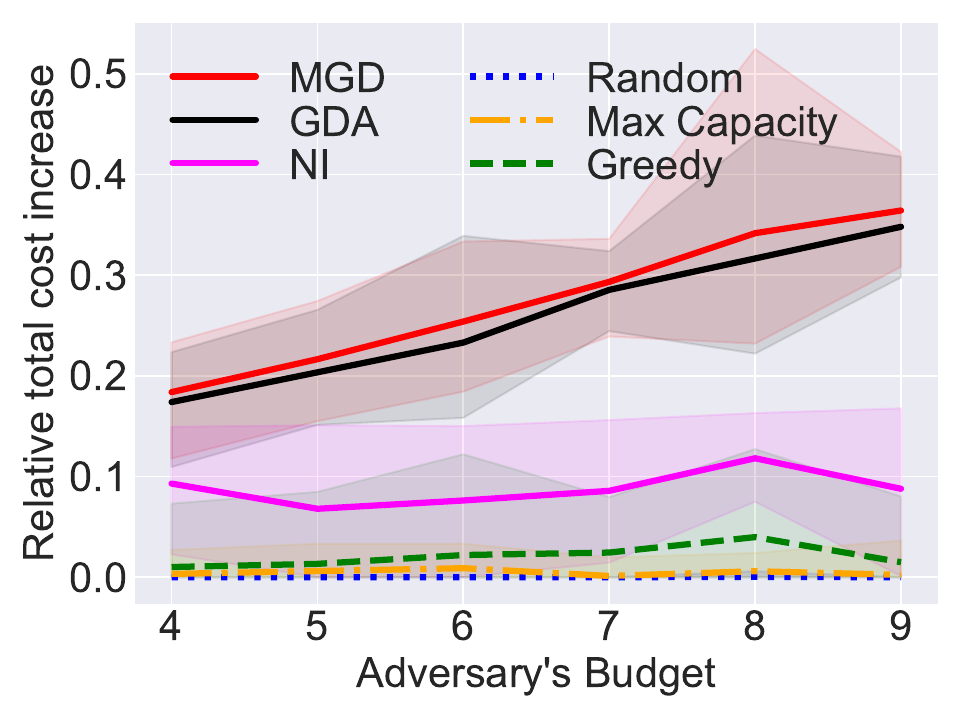}
\caption{$n=20$, $p=1$, $d=40\%$}\label{fig:cost_budget13}
\end{subfigure}
%
\hspace{4mm}
\begin{subfigure}[b]{0.35\textwidth}
\centering
\includegraphics[scale=0.35]{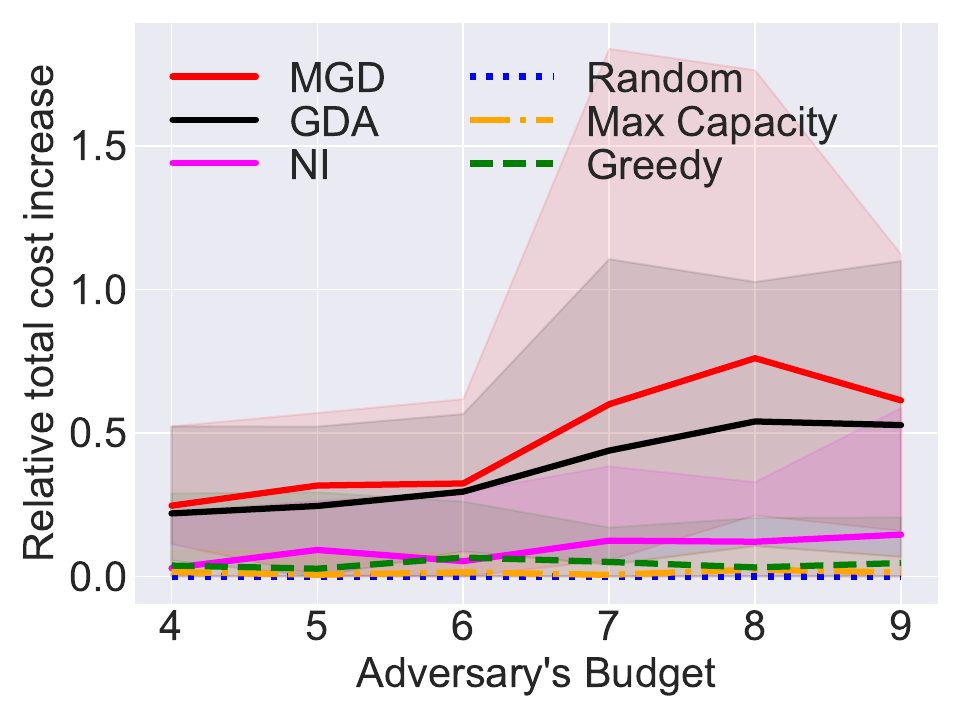}
\caption{$n=15$, $p=0.75$, $d=20\%$}\label{fig:cost_budget21}
\end{subfigure}
\\
\begin{subfigure}[b]{0.35\textwidth}
\centering
\includegraphics[scale=0.35]{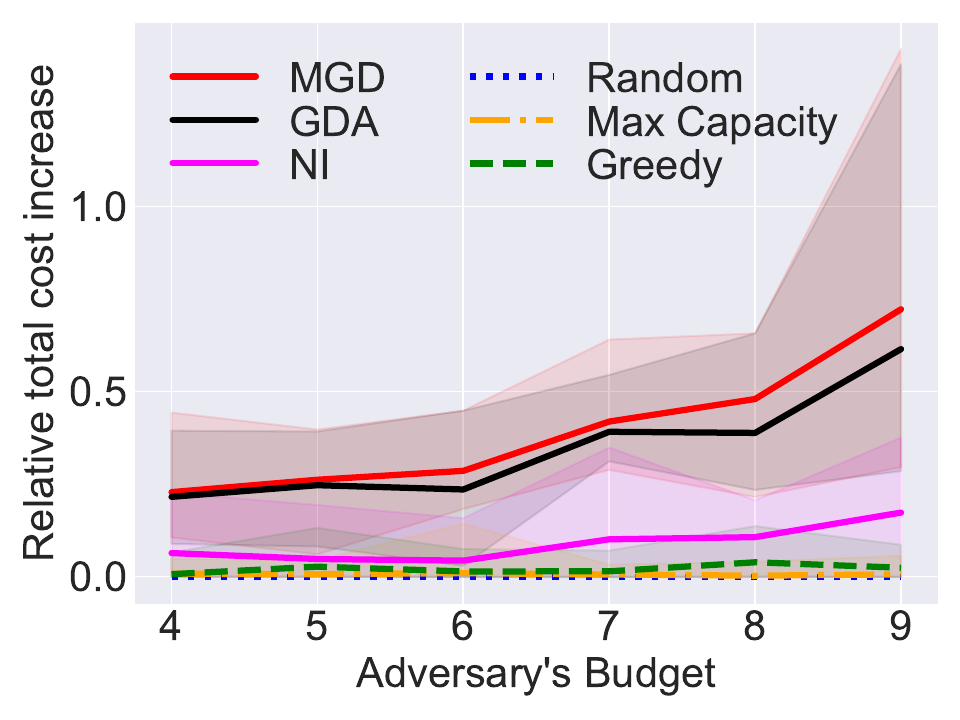}
\caption{$n=20$, $p=0.75$, $d=30\%$}\label{fig:cost_budget22}
\end{subfigure}
\hspace{4mm}
\begin{subfigure}[b]{0.35\textwidth}
\centering
\includegraphics[scale=0.35]{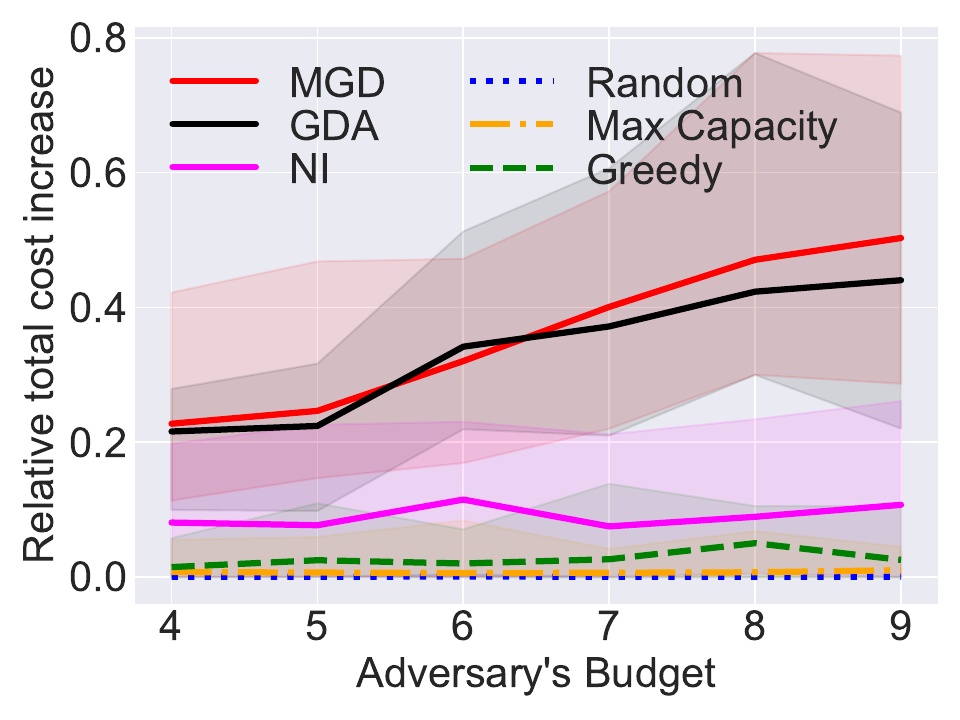}
\caption{$n=20$, $p=0.75$, $d=40\%$}\label{fig:cost_budget23}
\end{subfigure}
\caption{The evolution of the relative increase (due to the attack) of the total cost as a function of the budget of the adversary, for different settings. The range of the results (between minimum and maximum value across all runs) is also depicted using a shaded region around the average cost curve.}
\label{fig:min_cost_exp}
\end{figure}
 
\subsection{Effect of the Regularization Parameter}
We perform an additional set of experiments, in which we study the effect of the regularization parameter $\eta$ on the problem solution by examining the performance of the MGD algorithm. We would like to stress here that we select relatively small values of $\eta$, because the regularizer was introduced to impose the strong concavity assumption in $\y$ without significantly altering the original problem. In  Fig. \ref{fig}, we plot the evolution of the relative increase of the total cost as a function of the adversary's budget, for $\eta=10^{-8}, 10^{-4},10^{-2}$. We consider two different settings by changing the parameter $p$, while we only consider graphs with $n=15$ nodes; we also use the same stepsize and the same number of inner iterations as in the previous experiments. 

In Fig. \ref{fig} we see that it is not clear which value of $\eta$ leads to the best performance. Therefore, we cannot reach a clear conclusion about the effect of the size of the regularization parameter.  However, we also notice that the performance (such as the relative total cost increase) attained does not differ significantly among the different values of $\eta$, as long as the regularization parameter $\eta$ is small enough (such as the ones we test in our experiments, i.e., $\eta \leq 10^{-2}$). This behavior is desirable, since our goal is only to make the objective strongly concave in $\y$ in a way that the original problem (and thus its solutions) is not altered significantly.

\begin{figure}
  \centering
     \begin{subfigure}[t]{0.46\textwidth}
         \centering
     \includegraphics[width=1\textwidth]{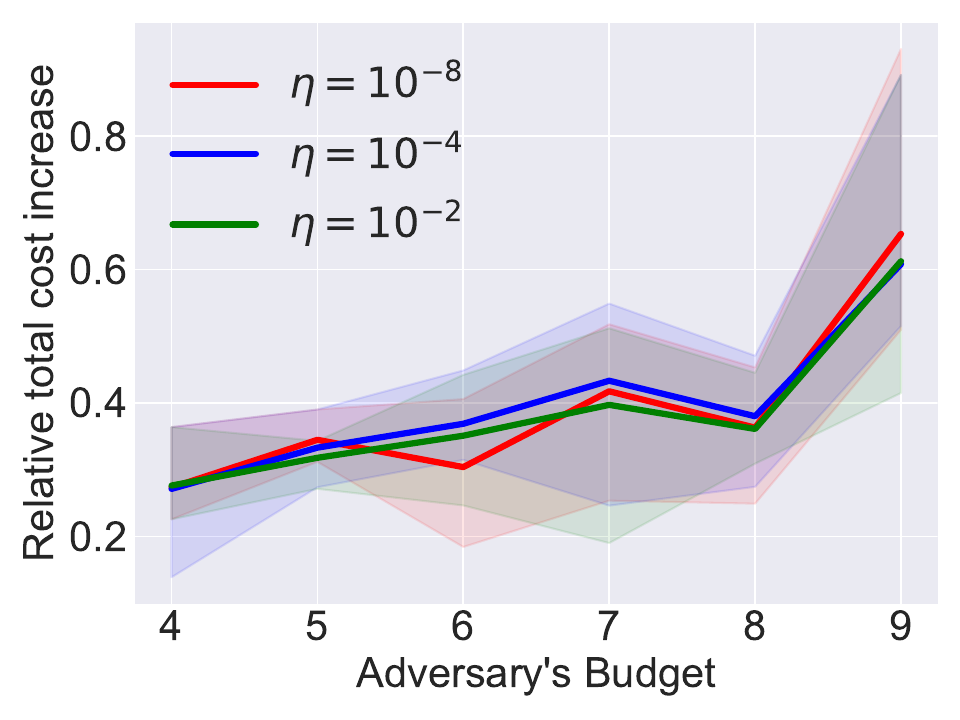} 
     \caption{$n=15, d=20\%,p=1$}
     \end{subfigure}  
     \begin{subfigure}[t]{0.46\textwidth}
         \centering
     \includegraphics[width=1\textwidth]{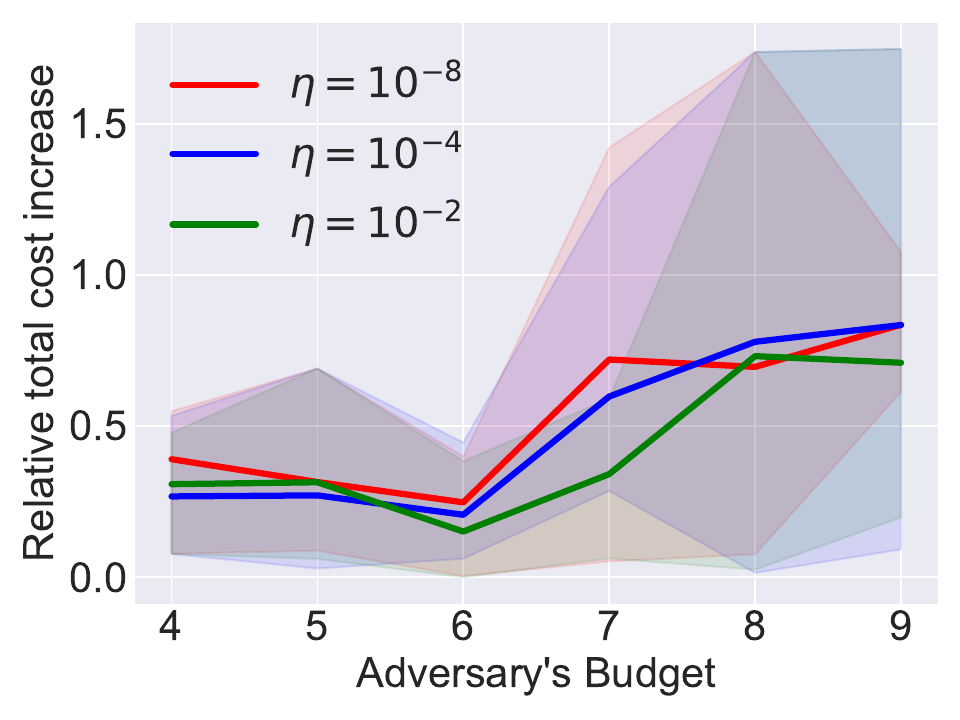} \caption{$n=15, d=20\%,p=0.75$}
     \end{subfigure}
    \caption{The evolution of the relative increase (due to the attack) of the total cost as a function of the budget of the adversary, for different values of the regularization parameter $\eta$. The range of the results (between the minimum and maximum value across all runs) is also depicted using a shaded region around the average cost curve.}
    \label{fig}
\end{figure}

\section{ Concluding Remarks}

In this paper we study a minimax problem, where the constraints of the inner maximization problem depends linearly on the variable of the outer one. This special structure makes the problem NP-hard, even when the objective is strongly-convex strongly-concave. We then develop the duality theory of the problem, and establish conditions under which strong duality holds. These results allow us to introduce  duality-based reformulations of the original problem,  and propose a set of first-order stationarity conditions.
Moreover, we demonstrate that the proposed formulation can model adversarial attacks on network flow problems.
In the future, we plan to study more applications of this formulation, as well as to related problems where the linear constraints are in the outer problem. 

It is possible to extend the main results (e.g., the duality theory, the definition and properties of stationary points) to problems with general, not necessarily linear, constraints of the form $g(\x,\y)\leq 0$, where $g:\R^{n} \times \R^{m} \to \R^{k}$. This is possible with addition of the following assumptions (below we provide only the key assumptions):
\begin{enumerate}
    \item $g_{i}(\cdot,\y)$ is concave $\forall \y \in \Y$ and $g_{i}(\x,\cdot)$ is convex $\forall \x \in \X, \forall i \in \{1,\ldots,k\}$. 
    \item For every $\xx$ there exists  $\y \in {\rm relint}(\Y)$ such that $g(\x,\y) < {\bf 0}$. 
    \item 
    $|g_i(\x_{1},\y_{1})-g_i(\x_{2},\y_{2})| \leq L\left[\|\x_{1}-\x_{2}\| + \|\y_{1}-\y_{2}\| \right],\; \forall~\x_{1},\x_{2}\in \mathcal{X}, ~\y_{1},\y_{2}\in \mathcal{Y}$, for some $L>0$, $\forall i \in \{1,\ldots,k\}$.
\end{enumerate}
Then, the duality theory developed in Theorems \ref{pro:weak_dual} and \ref{pro:strong_dual} will hold. Also, the definition of the (exact and approximate) stationary points in Definitions \ref{def:D2_stat} and \ref{def:D2_approx_stat}, as well as their respective properties in Propositions \ref{pro:stat}, \ref{pro:appr_stat}, will apply; we just need to substitute the linear constraints $A\x +B\y + \cb \leq 0$ with the more general ones $g(\x,\y)\leq 0$.
Finally, let us note that the NP-hardness results (Propositions \ref{eq:prob_hard} and \ref{eq:prob_hard2}) and the relations between the coupled constrained problems (Proposition \eqref{prop:prob_relat}) automatically apply in the case of the general constraints $g(\x,\y)\leq 0$.
Note that despite the ability to generalize our results, in this work we focus on linear constraints, since the changes required for this generalization are mainly technical and do not lead to new insights to the problem itself. 

\newpage
\appendix

\section{Proofs}\label{app:proof}

\subsection{Minimax problems with coupled linear constraints}\label{app:minimax}

\begin{proof}[\textbf{Proof of Proposition \ref{prop:prob_relat}}]
Below we provide the proof for the relations described in Proposition \ref{prop:prob_relat}.

\underline{$v \left( \mbox{\rm \ref{eq:inn_minimax}} \right) 
\leq v \left(\mbox{\rm \ref{eq:out_minimax}} \right) $}

For the function $\phi(\x): = \max\limits_{\y \in \Y, A\x+B\y\leq\bf{c}} f(\x,\y)$ defined in \eqref{eq:opt:inner2}, it holds that:
$$\phi(\x) \leq \max\limits_{\y \in \Y} f(\x,\y), \; \forall \xx.$$
This is because the RHS has a larger constraints set. Then, minimizing both sides of the above inequality, we obtain:
\begin{align}\label{eq:prf_rel_a}
\min\limits_{\x \in \X, A\x+B\y^*(\x) \leq \cb} \phi(\x) 
\leq \min\limits_{\x \in \X, A\x+B\y^*(\x) \leq \cb} \left( \max\limits_{\y \in \Y} f(\x,\y) \right),
\end{align}
where $\y^*(\x)$ is defined in \eqref{eq:out_minimax}.

Similarly, observe that the following holds:
\begin{align}\label{eq:prf_rel_b}
\min\limits_{\x \in \X } \phi(\x)
\leq \min\limits_{\x \in \X, A\x+B\y^*(\x) \leq \cb} \phi(\x). 
\end{align}
Combining \eqref{eq:prf_rel_a} and \eqref{eq:prf_rel_b} implies that:
\begin{align*}
\min\limits_{\x \in \X } \phi(\x)
&\leq \min\limits_{\x \in \X, A\x+B\y^*(\x) \leq \cb} \left( \max\limits_{\y \in \Y} f(\x,\y) \right) \\
\stackrel{\eqref{eq:opt:inner2}}{\Rightarrow} \min\limits_{\x \in \X } \left(\max\limits_{\y \in \Y, A\x+B\y\leq\bf{c}} f(\x,\y) \right)
&\leq 
\min\limits_{\x \in \X, A\x+B\y^*(\x) \leq \cb} \left( \max\limits_{\y \in \Y} f(\x,\y) \right).
\end{align*}
That is, $v \left( \mbox{\rm \ref{eq:inn_minimax}} \right) 
\leq v \left(\mbox{\rm \ref{eq:out_minimax}} \right) $.

To show that the strict inequality can also hold, consider the following problems, in which it is shown that $v \left(\mbox{\rm \ref{eq:inn_minimax}} \right) = -1$, and $v \left(\mbox{\rm  \ref{eq:out_minimax}} \right) = 1$:
\begin{align*}
\min\limits_{x \in [0,1]} \left( \max\limits_{y \in [0,1], x + y=1} x^{2}-y^{2} \right)
& = \min\limits_{x \in [0,1]} x^{2}-(1-x)^{2} 
= \min\limits_{x \in [0,1]} 2x - 1
= - 1,\\
\min\limits_{x \in [0,1], x + y^*(x)=1} \left( \max\limits_{y \in [0,1]} x^{2}-y^{2}\right) & = 
\min\limits_{x \in [0,1], x=1} x^{2} 
= 1.
\end{align*}
Therefore, in the above example it holds that $v \left(\mbox{\rm  \ref{eq:inn_minimax}} \right) <v \left(\mbox{\rm \ref{eq:out_minimax}} \right)$. This allows us to exclude the possibility that equality holds between \eqref{eq:inn_minimax} and \eqref{eq:out_minimax}, for all possible problem instances.

\underline{$v \left(\mbox{\rm  \ref{eq:inn_maximin}} \right) \geq v \left(\mbox{\rm  \ref{eq:out_maximin}} \right)$}

Let us define $\psi(\y): = \min\limits_{\xx, A\x+B\y\leq\bf{c}} f(\x,\y)$, 
it holds that 
$$\psi(\y) \geq \min\limits_{\xx} f(\x,\y), \; \forall \yy,$$
since the RHS has a larger constraints set. Then, we can maximize both sides of the above inequality, and obtain:
\begin{align}\label{eq:prf_rel_c}
\max\limits_{\yy, A\x^*(\y)+B\y \leq \cb} \psi(\y) 
\geq \max\limits_{\yy,A\x^*(\y)+B\y \leq \cb} \left( \min\limits_{\xx} f(\x,\y) \right),
\end{align}
where $\x^*(\y)$ is defined in \eqref{eq:out_maximin}.
In addition, we can easily see that 
\begin{align}\label{eq:prf_rel_d}
\max\limits_{\yy} \psi(\y)
\geq \max\limits_{\yy, A\x^*(\y)+B\y \leq \cb} \psi(\y). 
\end{align}
Combining \eqref{eq:prf_rel_c} and \eqref{eq:prf_rel_d} we obtain:
\begin{align*}
\max\limits_{\yy} \psi(\y)
&\geq \max\limits_{\yy,A\x^*(\y)+B\y \leq \cb} \left(\min\limits_{\xx} f(\x,\y)\right) \\
\Rightarrow \max\limits_{\yy} \left(\min\limits_{\xx, A\x+B\y\leq\bf{c}} f(\x,\y) \right)
&\geq 
\max\limits_{\yy,A\x^*(\y)+B\y \leq \cb} \left( \min\limits_{\xx} f(\x,\y) \right).
\end{align*}
That is, $v \left(\mbox{\rm  \ref{eq:inn_maximin}} \right) \geq v \left(\mbox{\rm  \ref{eq:out_maximin}} \right)$.

To show that the strict inequality can also hold, consider the following problems, in which it is shown that $v \left(\mbox{\rm  \ref{eq:inn_maximin}} \right) = 1$, and $v \left( \mbox{\rm \ref{eq:out_maximin}} \right) = -1$:
\begin{align*}
\max\limits_{y \in [0,1]} \left(\min\limits_{x \in [0,1], x + y=1} x^{2}-y^{2} \right)
& = \max\limits_{y \in [0,1]} (1-y)^{2}-y^{2} 
= \max\limits_{y \in [0,1]} -2y + 1
= 1,\\
\max\limits_{y \in [0,1], x^*(y) + y=1} \left( \min\limits_{x \in [0,1]} x^{2}-y^{2} \right) & = 
\max\limits_{y \in [0,1], y=1} -y^{2} 
= -1.
\end{align*}

In the above example it holds that $v \left( \mbox{\rm \ref{eq:inn_maximin}}\right) > v \left( \mbox{\rm \ref{eq:out_maximin}} \right)$. This allows us to exclude the possibility that equality can hold between \eqref{eq:inn_maximin} and \eqref{eq:out_maximin}, for all possible problem instances.

\underline{$v \left(\mbox{\rm  \ref{eq:out_minimax}} \right) \geq v \left( \mbox{\rm \ref{eq:out_maximin}} \right)$}

To begin with, observe that 
\begin{align}\label{eq:prf_rel_1}
\min\limits_{\x \in \X, A\x+B\y^*(\x) \leq \bf{c}} 
\left(\max\limits_{\y \in \Y} f(\x,\y) \right) 
\geq
\min\limits_{\x \in \X} 
\left(\max\limits_{\y \in \Y} f(\x,\y) \right), 
\end{align}
since the outer problem in the RHS has a larger constraints set. Further, the max-min inequality implies that:
\begin{align}\label{eq:prf_rel_2}
\min\limits_{\x \in \X} 
\left(\max\limits_{\y \in \Y} f(\x,\y) \right)
\geq 
\max\limits_{\y \in \Y} 
\left(\min\limits_{\x \in \X} f(\x,\y) \right).
\end{align}
Moreover, we have that: 
\begin{align}\label{eq:prf_rel_3}
\max\limits_{\y \in \Y} 
\left(\min\limits_{\x \in \X} f(\x,\y) \right)
\geq
\max\limits_{\y \in \Y, A\x^*(\y)+B\y \leq \bf{c}} 
\left(\min\limits_{\x \in \X} f(\x,\y) \right).
\end{align}
Combining equations \eqref{eq:prf_rel_1}, \eqref{eq:prf_rel_2},  \eqref{eq:prf_rel_3}, we obtain the following relation:
\begin{align}
\min\limits_{\x \in \X, A\x+B\y^*(\x) \leq \bf{c}} 
\left(\max\limits_{\y \in \Y} f(\x,\y) \right) 
\geq
\max\limits_{\y \in \Y, A\x^*(\y)+B\y \leq \bf{c}} 
\left(\min\limits_{\x \in \X} f(\x,\y) \right).
\end{align}
That is,  $v \left(\mbox{\rm  \ref{eq:out_minimax}} \right) \geq v \left( \mbox{\rm \ref{eq:out_maximin}} \right)$.


\underline{$v \left(\mbox{\rm  \ref{eq:inn_minimax}} \right), v \left( \mbox{\rm \ref{eq:out_maximin}} \right)$}

To begin with, for $A=B=c=0$ problems \eqref{eq:inn_minimax} and \eqref{eq:out_maximin} reduce to `classical' (w/o coupled constraints) minimax problems. The max-min inequality implies that $v \left( \mbox{\rm \ref{eq:inn_minimax}} \right)\geq v \left( \mbox{\rm \ref{eq:out_maximin}} \right)$.   

Moreover, consider the following problems, in which it is shown that  $v \left( \mbox{\rm \ref{eq:inn_minimax}} \right) = -3$ and $v \left( \mbox{\rm \ref{eq:out_maximin}} \right) = - 1$:
\begin{align*}
   \min\limits_{x \in [-1,1]} \left(\max\limits_{y \in [0,2], x + y=1} x^{2}-y^{2}\right)
& = \min\limits_{x \in [-1,1]} x^{2} - (1-x)^{2}
= \min\limits_{x \in [-1,1]}  2x - 1
= -3,\\ \max\limits_{y \in [0,2], x^*(y) + y=1} \left(\min\limits_{x \in [-1,1]} x^{2}-y^{2}\right) & = 
\max\limits_{y \in [0,2], y=1} -y^{2} 
= -1.
\end{align*}

The above examples show that there are problem instances where the relationship $v \left( \mbox{\rm \ref{eq:inn_minimax}} \right)< v \left( \mbox{\rm \ref{eq:out_maximin}} \right)$ holds between the two problems. 

\underline{$v \left( \mbox{\rm \ref{eq:inn_maximin}} \right)$, $v \left( \mbox{\rm \ref{eq:out_minimax}} \right)$}

To begin with, for $A=B=c=0$ problems \eqref{eq:inn_maximin} and \eqref{eq:out_minimax} reduce to the `classical' (w/o coupled constraints) minimax problems. The max-min inequality implies that $v \left(\mbox{\rm  \ref{eq:inn_maximin}} \right) \leq v \left( \mbox{\rm \ref{eq:out_minimax}} \right)$.   

Moreover, consider the following problems, in which it is shown that  $v \left( \mbox{\rm \ref{eq:inn_maximin}} \right) = 3$ and $v \left( \mbox{\rm \ref{eq:out_minimax}} \right) = 1$:
\begin{align*}
    \max\limits_{y \in [-1,0]} \left(\min\limits_{x \in [1,2], x + y=1} x^{2}-y^{2}\right)
&= \max\limits_{y \in [-1,0]} (1-y)^{2} - y^{2}
= \max\limits_{y \in [-1,0]} - 2y + 1
=3\\
\min\limits_{x \in [1,2], x + y^*(x)=1} \left( \max\limits_{y \in [-1,0]} x^{2}-y^{2}\right) & = 
\min\limits_{x \in [1,2], x=1} x^{2} 
= 1.
\end{align*}

The above examples show that there are problem instances where the relationship $v \left(\mbox{\rm \ref{eq:inn_maximin}}\right) > v \left(\mbox{\rm  \ref{eq:out_minimax}}\right)$ is realized between the two problems. 

\underline{$v \left(\mbox{\rm  \ref{eq:inn_minimax}} \right), v \left( \mbox{\rm \ref{eq:inn_maximin}}\right)$}

To begin with, for $A=B=c=0$ problems \eqref{eq:inn_minimax} and \eqref{eq:inn_maximin} reduce to the `classical' (w/o coupled constraints) minimax problems. The max-min inequality holds, that is $v \left(\mbox{\rm  \ref{eq:inn_minimax}} \right) \geq v \left(\mbox{\rm  \ref{eq:inn_maximin}} \right)$. 

Moreover, consider the following problems, in which it is shown that $v \left(\mbox{\rm  \ref{eq:inn_minimax}} \right) =-1$ and  $v \left(\mbox{\rm  \ref{eq:inn_maximin}} \right) = 1$:
    \begin{align}
     \min\limits_{x \in [0,1]} \left(\max_{y \in [0,1], x + y = 1} x^2-y^{2} \right)
    & =\min\limits_{x \in [0,1]} x^2 - (1-x)^2
    =\min\limits_{x \in [0,1]} -1 + 2x
    =-1\\
    \max\limits_{y \in [0,1]} \left(\min_{x \in [0,1], x + y = 1} x^2-y^{2} \right)
    & =\max\limits_{y \in [0,1]} (1-y)^2-y^2 
    =\max\limits_{y \in [0,1]} 1 - 2y
    =1.
    \end{align}
The above examples show that there are problem instances where the relationship $v \left(\mbox{\rm \ref{eq:inn_minimax}}\right) < v \left( \mbox{\rm \ref{eq:out_maximin}}\right)$ is realized between the two problems. 
\end{proof}


\subsection{First-order stationarity conditions}\label{app:solution}

In the proof of Lemma \ref{lem:inminmax_grad} we make use of Danskin's theorem \cite[Corollary, pg. 1167]{bernhard1995danskin}. Below we provide its formal statement.
\begin{lemma}[Danskin's Theorem]
Consider the problem $\phi(\x) = \max_{\y \in \Y} f(\x,\y)$, where $f:\R^{n} \times Y \to \mathbb{R}$ is differentiable and strongly concave in $\y$, and $\Y \subseteq \R^{m}$ is a non-empty compact set. Also, consider the set $\Y^{\ast}(\x) = \arg\max_{\y \in \Y} f(\x,\y)$ which is a singleton for any $\x \in \R^{n}$. Then, $\phi(\x)$ is differentiable and it holds that $\nabla \phi(\x) = \nabla_{x} f(\x,\y^{\ast}(\x))$, where $\y^{\ast}(\x) = \arg\max_{\y \in \Y} f(\x,\y)$.
\end{lemma}

\begin{proof}[\textbf{Proof of Lemma \ref{lem:inminmax_grad}}]

To begin with, we will compute the gradient of function $H(\x,\lb) = \max_{\yy}  L(\x,\y,\lb)$, defined in \eqref{eq:H}.
Indeed, notice that $\Y$ is a compact set, $L$ is differentiable and has a unique maximum due to its strong-concavity in $\y$. Then, from Danskin's theorem \cite{bernhard1995danskin} we can infer that $H(\x,\lb)$ is differentiable with:
\begin{align}\label{eq:H_grad}
    \nabla H(\x,\lb) 
= \g L\left(\x,\yo\left(\x,\lb\right), \lb\right) 
= \begin{bmatrix} \gx f(\x,\yo\left(\x,\lb\right)) - A^{T}\lb\\ 
-A\x -B \yo\left(\x,\lb\right) +\cb \end{bmatrix}, 
\end{align}
where $\yo(\x,\lb) := \arg\max\limits_{\yy} L(\x,\y,\lb)$.

Next we will show that  $H(\x,\lb)$ is strongly-convex in $\x$, with modulus $\mu_x$. Indeed for any $\x_1,\x_2 \in \X, \; \rho \in [0,1]$ we have the following series of relations:
\begin{align*}
&H(\rho\x_1 + (1-\rho)\x_2,\lb) 
= \max_{\yy} \left\{L(\rho\x_1 + (1-\rho)\x_2,\y, \lb)\right\} \\
& \stackrel{(a)}{=} \max_{\yy} \left\{ f(\rho \x_1 + (1-\rho ) \x_2, \y) - \lb^{T}\left(A\left(\rho \x_1 + (1-\rho ) \x_2\right) + B\y- \cb \right)\right\} \\
&\stackrel{(b)}{\leq} \max_{\yy} \bigg\{ \rho f(\x_1, \y) + (1-\rho)f(\x_2, \y)
 - \frac{\mu_{x}}{2} \rho (1 - \rho) \|\x_1 - \x_2 \|^{2}  \\
& \quad -\rho \lb^{T} A\x_1 - (1-\rho ) \lb^{T} A \x_2 -  \lb^{T}B \y+ \lb^{T}\cb \bigg\} \\
&\stackrel{(c)}{\leq} \max_{\yy} \left\{ \rho f(\x_1, \y) - \rho \lb^{T}\left(A\x_1 + B \y -\cb\right)\right\} - \frac{\mu_{x}}{2} \rho (1 - \rho) \|\x_1 - \x_2 \|^{2}  \\
& \quad+\max_{\yy} \left\{ (1-\rho)f(\x_2, \y) - (1-\rho)\lb^{T} \left( A\x_2 +B \y -\cb \right) \right\} 
 \\
& \leq \rho H(\x_1,\lb) + (1-\rho)H(\x_2,\lb) - \frac{\mu_{x}}{2} \rho (1 - \rho) \|\x_1 - \x_2 \|^{2},
\end{align*}
where in (a) we use the definition of Lagrangian from equation \eqref{eq:L}, in (b) the inequality follows from the strong convexity of $f$ w.r.t.\ $\x$ (with modulus $\mu_x$), and in (c) the triangle inequality of the max operator was used.

Moreover, note that:
\begin{align}\label{eq:pf_innminmax_1}
G(\lb) = \min\limits_{\xx}  \max\limits_{\yy}  L(\x,\y,\lb)
= \min\limits_{\xx} H(\x,\lb)
= - \max\limits_{\xx} \{ -H(\x,\lb)\}. 
\end{align}
Using the fact that $\X$ is a compact set, $-H$ is differentiable and strongly concave in $\x$, we can apply the Danskin's theorem \cite{bernhard1995danskin} and obtain:
\begin{align*}
\g G(\lb) &= 
\g_{\lb} H\left(\xo\left(\lb\right),\lb\right) =-A\xo(\lb) - B\yo\left(\xo\left(\lb\right),\lb\right) + \cb,
\end{align*}
where in the last equality we used the gradient of function $H(\x,\lb)$ from equation \eqref{eq:H_grad}, and  the definition of $\xo(\lb) = \arg\min\limits_{\xx}  H(\x,\lb)$.
The proof is now completed.
\end{proof}

\begin{proof}[\textbf{Proof of Proposition \ref{pro:stat}}]
Let $\lo \geq \bf{0}$ be a stationary point of $G(\lb)$, that is: $$\left\langle \nabla G\left(\lo\right),\lb-\lo \right\rangle \geq 0, \quad \forall \lb \geq 0.$$
Substituting the formula of the gradient of $G(\lb)$ from Lemma \ref{lem:inminmax_grad} gives:
\begin{align}\label{eq:stat_prf}
  \left\langle -A\xo\left(\lo\right) - B\yo\left(\xo\left(\lo\right),\lo\right) + \cb,\lb-\lo\right\rangle \geq 0, \quad \forall \lb \geq 0.  
\end{align}%

We will first show that these conditions imply that 
$A\xo\left(\lo\right) + B\yo\left(\xo\left(\lo\right),\lo\right) \leq \cb$. 
Suppose that there exists index $i \in \mathcal{K}$ such that $\left[-A\xo\left(\lo\right) - B\yo\left(\xo\left(\lo\right),\lo\right) + \cb\right]_{i} < 0$. Then, if we select $\lambda_{i} = \overline{\lambda}_{i} + \epsilon$, for some $\epsilon>0$, and $\lambda_{j} = \overline{\lambda}_{j}, \forall j \neq i$, \eqref{eq:stat_prf} becomes:
\begin{align*}
&\left\langle -A\xo\left(\lo\right) - B\yo\left(\xo\left(\lo\right),\lo\right) + \cb,\lb-\lo\right\rangle \\
&= \left[-A\xo\left(\lo\right) - B\yo\left(\xo\left(\lo\right),\lo\right) + \cb\right]_{i} \left(\overline{\lambda}_{i} + \epsilon -\overline{\lambda}_{i} \right) \\
&=\left[-A\xo\left(\lo\right) - B\yo\left(\xo\left(\lo\right),\lo\right) + \cb\right]_{i}\epsilon < 0.
\end{align*}
This is a contradiction.
As a result, $\left[-A\xo\left(\lo\right) - B\yo\left(\xo\left(\lo\right),\lo\right) + \cb\right]_{i} \geq 0, \forall i \in \mathcal{K}$. 

Next, we  show that the following holds:
\begin{align}\label{eq:slackness}
    \overline{\lambda_{i}} \left[-A\xo\left(\lo\right) - B\yo\left(\xo\left(\lo\right),\lo\right) + \cb\right]_{i} = 0, \quad \forall i \in \mathcal{K}.
\end{align}
 If  $\left[-A\xo\left(\lo\right) - B\yo\left(\xo\left(\lo\right),\lo\right) + \cb\right]_{i} = 0$, then the above equality holds trivially. 
Suppose that there exists $\left[-A\xo\left(\lo\right) - B\yo\left(\xo\left(\lo\right),\lo\right) + \cb\right]_{i} > 0$, such that $\overline{\lambda_{i}}>0$. Then, notice that we can select $\lambda_{i} = \overline{\lambda}_{i} - \epsilon$ (for some sufficiently small $\epsilon>0$) and $\lambda_{j} = \overline{\lambda}_{j}, \forall j \neq i$, substitute them into condition \eqref{eq:stat_prf}, and obtain:
\begin{align*}
&\left\langle -A\xo\left(\lo\right) - B\yo\left(\xo\left(\lo\right),\lo\right) + \cb,\lb-\lo\right\rangle \\
&= \left[-A\xo\left(\lo\right) - B\yo\left(\xo\left(\lo\right),\lo\right) + \cb\right]_{i}\left(\overline{\lambda}_{i} - \epsilon -\overline{\lambda}_{i}\right)\\
&= \left[-A\xo\left(\lo\right) - B\yo\left(\xo\left(\lo\right),\lo\right) + \cb\right]_{i}(-\epsilon) < 0,
\end{align*}
which is a contradiction to \eqref{eq:stat_prf}. This completes the proof of \eqref{eq:complementarity}.

From Definition \ref{def:D2_stat} we know that at a stationary point $\left(\xo,\yo,\lo\right)$ it holds that $\left(\xo,\yo\right)= \left( \xo\left(\lo\right),\yo\left(\xo,\lo\right) \right)$. 
Then, the optimality condition of $\yo$ follows directly from the definition of $\yo\left(\xo,\lo\right)$ [see \eqref{eq:argminmax}], i.e.,
\begin{align}\label{eq:stat_lem_y}
\yo = \yo\left(\xo,\lo\right)
= \arg\max_{\yy} L(\xo,\y,\lo). 
\end{align}
Similarly, from \eqref{eq:argminmax}, we have
$
\xo = \xo\left(\lo\right) = 
\arg\min\limits_{\xx} \left\{H\left(\x,\lo\right) \right\}. 
$
This concludes the proof.
\end{proof}

\begin{proof}[\textbf{Proof of Proposition \ref{pro:appr_stat}}]
Let $(\xt,\yt, \lt)$ be an $(\epsilon, \delta)$-approximate stationary solution. In order to compute a bound for the (magnitude of) violation of the constraint at $(\xt,\yt, \lt)$, we consider the following quantity, for an arbitrary $i \in \mathcal{K}$: 
\begin{align}\label{eq:prf_viol1}
\max\left\{0,\left[ A\xt+B\yt-\cb \right]_i \right\} 
&=\max\bigg\{0,\left[A\left(\xt-\xo\left(\lt\right)\right)+B\left(\yt-\yo\left(\xo\left(\lt\right),\lt\right)\right)\right]_i  \nonumber \\
&+ \left[A\xo\left(\lt\right)+B\yo\left(\xo\left(\lt\right),\lt\right) -\cb \right]_i \bigg\} \nonumber \\
&\leq \max\left\{0,\left[A\left(\xt-\xo\left(\lt\right)\right)+B\left(\yt-\yo\left(\xo\left(\lt\right),\lt\right)\right)\right]_i\right\} \nonumber\\
& + \max\left\{0, \left[A\xo\left(\lt\right)+B\yo\left(\xo\left(\lt\right),\lt\right) -\cb \right]_i\right\}.
\end{align}
We can bound the first term in the above inequality as follows:
\begin{align}\label{eq:prf_viol2}
    &\max\left\{0,\left[A\left(\xt-\xo\left(\lt\right)\right)+B\left(\yt-\yo\left(\xo\left(\lt\right),\lt\right)\right)\right]_i\right\} \nonumber\\
    &\leq
    \bigg| \left[A\left(\xt-\xo\left(\lt\right)\right)+B\left(\yt-\yo\left(\xo\left(\lt\right),\lt\right)\right)\right]_i \bigg| \nonumber \\
    &\stackrel{(a)}{\leq} \|A\left(\xt-\xo\left(\lt\right)\right)+B\left(\yt-\yo\left(\xo\left(\lt\right), \lt\right)\right)\| \nonumber \\
    &\stackrel{(b)}{\leq} \|A\|\|\xt-\xo\left(\lt\right) \| + \|B\| \|\yt-\yo\left(\xo\left(\lt\right), \lt\right) \| \nonumber\\
    & \stackrel{(c)}{\leq} \sigma_{\rm max} \cdot \left( \|\xt-\xo\left(\lt\right) \| + \|\yt-\yo\left(\xo\left(\lt\right),\lt\right) \| \right)\stackrel{(d)}{\leq} 2\sigma_{\rm max} \cdot \delta,
\end{align}
where in (a) we used the fact that for a vector $\x$ we have $|x_i| \leq \sqrt{\sum_{i =1}^{n} x_i^2}=\|\x\|$; in (b) we used the triangle inequality and the following property of norms: $\|A\x\| \leq \|A\| \|\x\|$; in (c) we used the definition: $\sigma_{\rm max} := \max\{\|A\|,\|B\|\}$; and in (d) we used the condition $d(\xt,\yt, \lt) = \|\xo\left(\lt\right) - \xt \|^2 + \|\yo\left(\xo\left(\lt\right),\lt\right) - \y \|^2 \leq \delta^{2}$ from Definition \ref{def:D2_approx_stat}, which implies that $\|\xo\left(\lt\right) - \xt \| \leq \delta$ and $\|\yo\left(\xo\left(\lt\right),\lt\right) - \y \| \leq \delta$.

To derive a bound for the  second term of \eqref{eq:prf_viol1}, recall that the  definition of approximate stationarity requires that $\|Q\left(\lt\right)\| \leq \epsilon$. This implies that:
\begin{align}\label{eq:prf_viol3}
    \epsilon& \ge \bigg| \left[ \frac{1}{\alpha} \left( \lt - \text{proj}_{\R^{k}_{+}}\left( \lt-\alpha\g G\left(\lt\right) \right) \right)\right]_{i} \bigg|  \nonumber\\
    & \stackrel{(i)}= \bigg| \frac{1}{\alpha} \left( \widetilde{\lambda}_{i} - \max\left\{0, \widetilde{\lambda}_{i}-\alpha\g_{i} G\left(\lt\right) \right\} \right)  \bigg| \nonumber\\
    & \stackrel{(ii)} = \bigg| \frac{1}{\alpha} \left( \widetilde{\lambda}_{i} -  \max\left\{0, \widetilde{\lambda}_{i}-\alpha  \left[-A\xo\left(\lt\right)-B\yo\left(\xo\left(\lt\right),\lt\right) +\cb\right]_{i} \right\} \right)  \bigg|
\end{align}
 where $(i)$ uses the fact that the projection is performed w.r.t. the set $\R^{k}_{+}$, thus it can applied component-wise; $(ii)$ uses the expression for $\nabla G\left(\lt\right)$ derived in Lemma \ref{lem:inminmax_grad}. 
 
Suppose that the $i$th constraint is violated, i.e., {\small $\left[-A\xo\left(\lt\right)-B\yo\left(\xo\left(\lt\right),\lt\right) +\cb\right]_{i} < 0$}. Then condition \eqref{eq:prf_viol3} takes the following form:
\begin{align}\label{eq:prf_viol4}
     &\bigg| \frac{1}{\alpha} \left( \widetilde{\lambda}_{i} - \widetilde{\lambda}_{i}+\alpha  \left[-A\xo\left(\lt\right)-B\yo\left(\xo\left(\lt\right),\lt\right) +\cb\right]_{i} \right)  \bigg| \leq \epsilon \nonumber\\ 
    \Rightarrow&\bigg|  \left[-A\xo\left(\lt\right)-B\yo\left(\xo\left(\lt\right),\lt\right) +\cb\right]_{i} \bigg| \leq \epsilon \nonumber\\ 
    \Rightarrow&\left[A\xo\left(\lt\right)+B\yo\left(\xo\left(\lt\right),\lt\right) -\cb\right]_{i} \leq \epsilon,
\end{align}
where the first relation uses the fact that $\tilde{\lambda}_i\ge 0$, so the argument inside the max function is positive.

Overall, in the case where $\left[-A\xo\left(\lt\right)-B\yo\left(\xo\left(\lt\right),\lt\right) +\cb\right]_{i} \geq 0$ holds, then \eqref{eq:prf_viol1} and \eqref{eq:prf_viol2} together imply that:
\begin{align*}
\max\left\{0,\left[ A\xt+B\yt-\cb \right]_i \right\} 
\leq 2 \sigma_{\rm max} \delta + 0 \leq 2 \sigma_{\rm max} \delta. 
\end{align*}

On the other hand, if $\left[-A\xo\left(\lt\right)-B\yo\left(\xo\left(\lt\right),\lt\right) +\cb\right]_{i} < 0$, then combining  \eqref{eq:prf_viol1}, \eqref{eq:prf_viol2} and \eqref{eq:prf_viol4} yields
$
\max\{0,\left[ A\xt+B\yt-\cb \right]_i \} 
\leq 2 \sigma_{\rm max} \delta + \epsilon.
$

Finally, we know from Definition \ref{def:D2_approx_stat} that the following condition holds at the approximate stationary point $\left(\xt,\yt,\lt\right)$:
$$d\left(\xt,\yt, \lt\right)  = \|\xo\left(\lt\right) - \xt \|^2 + \|\yo\left(\xo\left(\lt\right),\lt\right) - \yt \|^2 \leq \delta^{2}.$$ 
Therefore, it follows directly that 
$
\|\yo - \yt \| \leq \delta, \; \|\xo - \xt \|\leq \delta,
$
where $(\xo,\yo) = \left( \xo\left(\lt\right),\yo\left(\xo,\lt\right) \right)$. This completes the proof.
\vspace{-8mm}
\begin{align*}
\end{align*}

\end{proof}

\bibliographystyle{siamplain}
\bibliography{references}
\end{document}